\documentclass[12pt,reqno]{amsart}

\usepackage{etex}
\usepackage[utf8]{inputenc}
\usepackage{amsfonts}
\usepackage{amsmath}
\usepackage{amssymb}
\usepackage{amsthm}
\usepackage{mathrsfs}
\usepackage{stmaryrd}
\usepackage{color}
\usepackage[english]{babel}
\usepackage{fontenc}
\usepackage{url}
\usepackage{graphicx}
\usepackage{hyperref}
\usepackage{caption}
\usepackage{epstopdf}
\usepackage{hyphenat}
\usepackage{float}
\usepackage{indentfirst}
\usepackage{tabularx}
\usepackage[export]{adjustbox}
\usepackage{tikz}
\usepackage{color}
\usepackage[all]{xy}
\usetikzlibrary{matrix,arrows,decorations.pathmorphing}
\usepackage{booktabs}
\usepackage{array}
\newcolumntype{P}[1]{>{\centering\arraybackslash}p{#1}}

\floatstyle{plaintop}
\restylefloat{table}

\allowdisplaybreaks

\newtheorem{theorem}{Theorem}[section]
\newtheorem{lemma}[theorem]{Lemma}
\newtheorem{proposition}[theorem]{Proposition}
\newtheorem{corollary}[theorem]{Corollary}
\newtheorem{conjecture}[theorem]{Conjecture}
\newtheorem*{theorem*}{Theorem}

\theoremstyle{definition}

\newtheorem{observation}[theorem]{Observation}
\newtheorem{remark}[theorem]{Remark}

\newtheorem{example}[theorem]{Example}
\newtheorem{definition}[theorem]{Definition}

\newtheorem{question}[theorem]{Question}

\newtheorem*{general set up}{General set up}
\newtheorem*{description of the action}{Description of the action}
\newtheorem*{chambers decomposition}{Chambers decomposition}

\newcommand{\PP}{\mathbb{P}}
\newcommand{\rM}{\mathrm{M}}
\newcommand{\rN}{\mathrm{N}}
\newcommand{\cO}{\mathcal{O}}
\newcommand{\cT}{\mathcal{T}}

\newcommand{\spec}{\mathrm{Spec}}

\newcommand{\GL}{\mathrm{GL}}
\newcommand{\SL}{\mathrm{SL}}
\newcommand{\PGL}{\mathrm{PGL}}

\newcommand{\git}{/\!/}

\makeatletter
\def\l@subsection{\@tocline{2}{0pt}{2.5pc}{5pc}{}}
\makeatother

\author{Alessio Caminata}
\address{Institut de Matem\`{a}tica, Universitat de Barcelona \\ Gran Via de les Corts Catalanes 585, 08007 Barcelona, Spain}
\email{caminata@ub.edu} 

\author{Noah Giansiracusa}
\address{Department of Mathematics and Statistics, Swarthmore College, Swarthmore, PA 19081}
\email{ngiansi1@swarthmore.edu}

\author{Han-Bom Moon}
\address{Department of Mathematics, Fordham University, New York, NY 10023}
\email{hmoon8@fordham.edu}

\author{Luca Schaffler}
\address{Department of Mathematics and Statistics, University of Massachusetts Amherst, Amherst, MA 01003}
\email{schaffler@math.umass.edu}


\title{Equations for point configurations to lie on a rational normal curve}

\begin{document}

\maketitle

\begin{abstract}
The parameter space of $n$ ordered points in projective $d$-space that lie on a rational normal curve admits a natural compactification by taking the Zariski closure in~$(\PP^d)^n$.  The resulting variety was used to study the birational geometry of the moduli space~$\overline{\mathrm{M}}_{0,n}$ of $n$-tuples of points in~$\PP^1$. In this paper we turn to a more classical question, first asked independently by both Speyer and Sturmfels: what are the defining equations? For conics, namely $d=2$, we find scheme-theoretic equations revealing a determinantal structure and use this to prove some geometric properties; moreover, determining which subsets of these equations suffice set-theoretically is equivalent to a well-studied combinatorial problem.  For twisted cubics, $d=3$, we use the Gale transform to produce equations defining the union of two irreducible components, the compactified configuration space we want and the locus of degenerate point configurations, and we explain the challenges involved in eliminating this extra component.  For~$d \ge 4$ we conjecture a similar situation and prove partial results in this direction.
\end{abstract}



\section{Introduction}

Configuration spaces are central to many modern areas of geometry, topology, and physics.  Rational normal curves are among the most classical objects in algebraic geometry.  In this paper we explore a setting where these two realms meet: the configuration space of $n$ ordered points in~$\PP^d$ that lie on a rational normal curve.  This is naturally a subvariety of~$(\PP^d)^n$, and by taking the Zariski closure we obtain a compactification of this configuration space, which we call the \emph{Veronese compactification}~$V_{d,n} \subseteq (\PP^d)^n$.  

The Veronese compactification parameterizes configurations of (possibly coincident) points supported on a flat limit of a rational normal curve.  Such a flat limit, if it is non-degenerate, is a rational normal curve or a union of rational normal curves of lower degree where \'etale locally the components meet like the coordinate axes in affine space.  Our main focus is to find and study the multi-homogeneous equations cutting out the Veronese compactification.

\subsection{Results and proof outlines}

First, some notation and conventions.  We work over an algebraically closed field~$\Bbbk$ of arbitrary characteristic.  Let~$[n] := \{1, 2, \ldots, n\}$, and denote the set of cardinality $m$ subsets of $[n]$ by $\binom{[n]}{m}$.  An ordered configuration of $n$ points in~$\PP^d$ is usually written \[\mathbf{p} = (p_1,\ldots,p_n) \in (\PP^d)^n.\]  The Veronese compactification~$V_{d,n}$ is irreducible of dimension $d^2 + 2d + n - 3$ (see Lemma~\ref{lem:Vdnirreducible}) and equals~$(\PP^d)^n$ when $d=1$ or $n \le d+3$, so assume throughout that $d \ge 2$ and $n \ge d+4$.  

We can say the most in the case of conics, $d=2$.
\begin{theorem}
For~$n \ge 6$, we have:
\begin{enumerate}
\item $V_{2,n}$ is defined scheme-theoretically by~$\binom{n}{6}$ determinants of~$6\times 6$ matrices whose entries are quadratic monomials;
\item a subset~$\cT \subseteq \binom{[n]}{6}$ of these determinants defines $V_{2,n}$ set-theoretically if and only if for any partition $I_1\sqcup \cdots \sqcup I_6 = [n]$, there exists $J\in \cT$ such that $|J \cap I_j| = 1$ for all $1 \le j \le 6$---consequently, the number of these determinants that suffice set-theoretically is at least~$\frac{2}{n-4}\binom{n}{6}$;
\item $V_{2,n}$ is Cohen-Macaulay, normal, and it is Gorenstein if and only if~$n = 6$.
\end{enumerate}
\end{theorem}

The starting point here is the elementary observation that six points in~$\PP^2$ lie on a conic if and only if their images under the Veronese embedding $\PP^2 \hookrightarrow \PP^5$ lie on a hyperplane.  This yields a~$6 \times 6$ determinant defining the hypersurface~$V_{2,6} \subseteq (\PP^2)^6$.  We then provide an inductive argument showing that $V_{2,n} \subseteq (\PP^2)^n$ is defined by the pullback of this determinant along all~$\binom{n}{6}$ forgetful maps~$V_{2,n} \twoheadrightarrow V_{2,6}$.  We can then apply results from the theory of determinantal varieties to see that $V_{2,n}$ is Cohen-Macaulay and classify when it is Gorenstein; constructing an isomorphism in codimension one with the Kontsevich stable map space~$\overline{\rM}_{0,n}(\PP^{2},2)$ then yields normality.

The combinatorial property stated above in~(2) is a special case of a set transversality problem that has been studied by many authors (see~\cite{BT09} and the references cited therein).  We prove the first assertion in~(2) by using the Veronese embedding to reduce to the analogous problem of determining which minors of a~$(d+1)\times n$ matrix of homogeneous coordinates for~$(\PP^d)^n$ set-theoretically cut out the locus~$Y_{d,n} \subseteq (\PP^d)^n$ of degenerate point configurations---we solve this problem with a direct linear algebraic analysis.  The second assertion in~(2) then follows from general bounds discussed in~\cite{BT09}, though we also provide a new argument for a bound that is close to this one.

Let us turn now to~$d \ge 3$. Recall that the Gale transform is an involution that associates to $n$ general points in $\mathbb{P}^d$ a configuration of $n$ points in $\mathbb{P}^{n-d-2}$ defined up to the action of $\SL_{n-d-2}$. The locus~$V_{d,n} \subseteq (\PP^d)^n$ is $\SL_{d+1}$-invariant, and the Gale transform provides an involutive isomorphism of GIT quotients~\cite[Corollary III.1]{DO88} \[(\PP^d)^n\git \SL_{d+1} \cong (\PP^{n-d-2})^n\git \SL_{n-d-1}\] sending configurations supported on a rational normal curve of degree~$d$ to configurations supported on a rational normal curve of degree~$n-d-2$~\cite{Gop70,Gop84}.  Our basic idea is the following.  We have the equations for~$V_{2,d+4}$, and the Gale transform sends these to equations that are satisfied on~$V_{d,d+4}$; by pulling these back along the~$\binom{n}{d+4}$~forgetful maps, we get equations that are satisfied on~$V_{d,n}$.    Let us denote by~$W_{d,n} \subseteq (\PP^d)^n$ the subscheme these latter equations define. Moreover, we denote by~$Y_{d,n} \subseteq (\PP^d)^n$ the determinantal variety parameterizing degenerate point configurations, i.e. $Y_{d,n}$ is defined by all~$(d+1) \times (d+1)$~minors of the~$(d+1)\times n$~matrix whose columns are given by the homogeneous coordinates of each copy of~$\PP^d$.

\begin{theorem}
For~$d \ge 3$, let $W_{d,n}$ and $Y_{d,n}$ be the schemes discussed above.  Then:
\begin{enumerate}
\item scheme-theoretically, we have $V_{d,n}\cup Y_{d,n} \subseteq W_{d,n}$;
\item the above inclusion is a set-theoretic equality if~$d=3$ or~$n = d+4$ (we conjecture that equality always holds).
\end{enumerate}
\end{theorem}

Thus we have explicit set-theoretic determinantal equations for the variety \[W_{3,n} = V_{3,n} \cup Y_{3,n}.\]  In the first two non-trivial cases, $V_{3,7}$ and $V_{3,8}$, we prove that every degenerate configuration lies on a flat limit of a twisted cubic (i.e., $Y_{3,7} \subseteq V_{3,7}$ and $Y_{3,8} \subseteq V_{3,8}$), so we actually get equations for these Veronese compactifications (we also get equations for~$V_{4,8}$ in this way).  For~$n \ge 9$, however, $W_{3,n}$ really does have two irreducible components.  Essentially, this is a consequence of the Gale transform not being defined on degenerate configurations.  

We prove three statements that reveal some of the considerable challenges that occur here which did not occur for~$d=2$: (1) $V_{3,n}$ for~$n\ge 9$ is not defined, even set-theoretically, by pulling back the equations for~$V_{3,7}$; (2) $V_{3,n}$ is not normal for~$n \ge 8$; (3) the polynomials that cut~$W_{d,n}$ down to the single irreducible component~$V_{d,n}$ cannot be $\PGL_{d+1}$-invariant.  This third property rules out many geometric constructions for producing equations, while the first property suggests the complexity of the equations increases as~$n$ increases.

\begin{remark}
There are other possible approaches to finding equations for~$V_{d,n}$.  For instance, there are nice expressions for the equations of a rational normal curve in certain standard forms, so one general strategy is to first use the~$\PGL_{d+1}$-action to move an arbitrary point of~$V_{d,n}$ into a standard form and then apply these well-known equations.  However, this introduces a large number of extra variables, the entries of the~$\PGL_{d+1}$~matrix, and the necessary elimination theory appears to get quite complicated quickly.  The Gale transform takes advantage of the fact that $V_{d,n}$ is $\PGL_{d+1}$-invariant and this symmetry significantly cuts down the complexity of the equations.
\end{remark}

\subsection{Background and context}

Our main motivation for embarking on this study is simply that rational normal curves are such ubiquitous classical objects that it seems very natural to ask for the equations describing the locus of point configurations supported on them.  In fact, Speyer asked precisely this question in a MathOverflow post~\cite{Spe14}, and in personal correspondence Sturmfels had asked the second author the same question.  

The condition that points lie on a rational normal curve arises in a variety of settings---for example, six linearly general points in~$\PP^2$ lie on a conic if and only if they are self-associated by the Gale transform~\cite{Cob22} (see also~\cite[Example 2.3.12]{Kap93} for a modern discussion); seven linearly general points in~$\PP^3$ lie on a twisted cubic if and only if they do not satisfy the Minimal Resolution Conjecture~\cite[Proposition 3.3]{CRV93}; the Hilbert function of a configuration of fat points in~$\PP^d$ conjecturally does not depend on their locations if they lie on a rational normal curve~\cite{CEG99}; the blow-up of~$\PP^d$ at any number of points is a Mori Dream Space if the points lie on a rational normal curve~\cite[Theorem 1.2]{CT06}; if~$p_0,\ldots,p_d$ and $q_0,\ldots,q_d$ are two bases for~$\PP^d$, each apolar with respect to a fixed non-degenerate quadratic form, and the codimension of the space of degree~$d$ hypersurfaces with multiplicity~$d-2$ at the~$p_i$ and passing through the~$q_i$ is at least two in the space of degree~$d$ hypersurfaces with multiplicity~$d-2$ at the~$p_i$, then all the~$p_i$ and~$q_i$ lie on a rational normal curve~\cite[Theorem 1.8]{Lan99}.

There are several instances where explicit equations for a moduli space were worked out and used to better understand that moduli space.  Most closely related to the topic of this paper, we have Keel and Tevelev's equations for the Grothendieck-Knudsen compactification~$\overline{\rM}_{0,n}$ of the moduli space of~$n$~points on the line, in its log canonical embedding~\cite{KT09} (see also~\cite{MR17}), and Howard, Millson, Snowden, and Vakil's equations for the GIT compactifications of this same moduli space, in their natural invariant-theoretic projective embeddings~\cite{HMSV09}; also, Ren, Sam, and Sturmfels studied certain 19th-century moduli spaces and used their explicit equations to explore tropicalizations~\cite{RSS14}.  

While the Veronese compactification~$V_{d,n}$ is a parameter space rather than a moduli space (meaning that we do not quotient out by automorphisms), the results in this paper may lead to progress related to the investigations above. For instance, the GIT quotients~$V_{d,n}\git \SL_{d+1}$ were used in~\cite{Gia13,GG12,GJM13} to study the birational geometry of~$\overline{\rM}_{0,n}$, and since $Y_{d,n}$ is unstable for all GIT linearizations, our equations for~$W_{d,n}$ yield equations for \[V_{d,n}\git \SL_{d+1} \subseteq (\PP^d)^n\git \SL_{d+1}\] (definitely for~$d=2,3$, and conjecturally for~all~$d$).  The Howard-Millson-Snowden-Vakil results, while solving a 100-year-old problem, are only the~$d=1$~case of their program, and perhaps our equations for~$V_{d,n}\git \SL_{d+1}$ will help illuminate the unknown equations for~$(\PP^d)^n\git \SL_{d+1}$.  Moreover, the tropicalization of~$V_{2,n}$ is a compactification of the image of the space of phylogenetic trees under the 3-dissimilarity map embedding~\cite{PS04}, so our equations for the former could be used to study the latter.

We defer these potential applications to future investigations and focus here on the structure of the equations for~$V_{d,n}$ and what they imply geometrically.  While our results provide important first steps in exploring this elementary yet richly complex object from an explicit equational perspective, there is clearly more to do and the story will only become more engrossing as the higher~$d$~cases are fully unraveled.  

\subsection*{Acknowledgement}
We thank Bernd Sturmfels for introducing us to this problem and we thank Aldo Conca, Yaim Cooper, Igor Dolgachev, Sam Grushevsky, Seung Jin Lee, Julie Rana, David Speyer, Bernd Sturmfels, Jenia Tevelev, Dennis Tseng, and Matteo Varbaro for helpful conversations regarding it. A special thanks to David Speyer for pointing out a minor mistake in a previous version of this paper, which is now corrected. We also thank the anonymous referee for carefully reading our paper and for the valuable comments. The first author was supported by the European Union's Horizon~2020 research and innovation programme under grant agreement No.~701807; the second author was partially supported by the NSA Young Investigator Grant~H98230-16-1-0015; the third author was partially supported by the Minerva Research Foundation; the fourth author was partially supported by the NSF grant~DMS-1603604 and the Research and Training Group in Algebra, Algebraic Geometry, and Number Theory as a student at the University of Georgia. 


\section{Rational normal curves, configurations, and their limits}

This section contains preliminary material that we rely on in the rest of the paper.  We introduce the main object of interest, a compactification of the space of configurations of distinct points supported on a rational normal curve, see Definition~\ref{definitionUdnVdn}.  We first define this as a closure in the Zariski topology, then we explain how it can also be constructed as the image of a Kontsevich stable map moduli space under a natural morphism, see Proposition~\ref{prop:basicprops}.  With the help of a lemma of Artin, we obtain a useful description of the limit curves supporting configurations in the boundary of this compactification: essentially, they are unions of rational normal curves intersecting nicely, see Proposition~\ref{degeneratepointconfigonquasiVeronese}. 

Recall that a \emph{rational normal curve} in~$\PP^{d}$ is a smooth rational curve of degree~$d$.  Up to projective automorphism there is a unique rational normal curve in~$\PP^d$, namely, the image of the~$d$-th~Veronese map~$\PP^1 \hookrightarrow \PP^d$.   Two well-known and useful classical facts are: 
\begin{enumerate}
\item any $d+3$ points in~$\PP^{d}$ in linearly general position lie on a unique rational normal curve~\cite[Theorem 1.18]{Har95} (this is Castelnuovo's Lemma), and
\item distinct points on a rational normal curve are in linearly general position (this follows from basic properties of the Van der Monde determinant).
\end{enumerate}

\begin{definition}
\label{definitionUdnVdn}
Let $d$ and $n$ be positive integers. 
\begin{enumerate}
\item Let~$U_{d, n} \subseteq (\PP^{d})^{n}$ be the subvariety parameterizing configurations~$\mathbf{p} = (p_{1},  \ldots, p_{n})$ of~$n$ distinct points in~$\PP^d$ such that there exists a rational normal curve~$C \subseteq \PP^d$ with~$p_i \in C$ for all~$i=1,\ldots,n$. 
\item Let~$V_{d, n} \subseteq (\PP^{d})^{n}$ be the Zariski closure of~$U_{d, n}$, equipped with the reduced induced scheme structure.  We call this the \emph{Veronese compactification} of~$U_{d, n}$. 
\end{enumerate}
\end{definition}

Castelnuovo's Lemma implies that $V_{d, n} = (\PP^{d})^{n}$ whenever $n \le d+3$. Clearly $V_{1, n} = (\PP^{1})^{n}$, so from now on we assume $d \ge 2$. 

\begin{lemma}\label{lem:Vdnirreducible}
For~$n \ge d+3$, the variety~$V_{d, n}$ is irreducible of dimension $d^{2}+2d+n-3$. 
\end{lemma}

\begin{proof}
Let~$\iota \colon \PP^{1} \hookrightarrow \PP^{d}$ be a fixed rational normal curve.  Consider the map 
\begin{eqnarray*}
	\psi \colon \PGL_{d+1} \times (\PP^{1})^{n} & \to & V_{d, n}\\
	(g, (x_{1}, \ldots, x_{n})) & \mapsto & 
	((g \circ \iota)(x_{1}), \ldots, 
	(g \circ \iota)(x_{n})).
\end{eqnarray*}
Since there is only one rational normal curve up to projective equivalence, we have \[U_{d, n} \subseteq \mathrm{im}\; \psi \subseteq \overline{U}_{d,n} = V_{d,n}.\] Since the domain of~$\psi$ is irreducible, we know that $\mathrm{im}\; \psi$, and hence $V_{d,n}$, is irreducible.

For a general~$\mathbf{p} \in U_{d, n}$, by Castelnuovo's Lemma there is a unique rational normal curve~$C \subseteq \PP^{d}$ that contains $\mathbf{p}$. Let $h \in \PGL_{d+1}$ such that $\mathrm{im} (h \circ \iota) = C$. Then \[\psi^{-1}(\mathbf{p}) = \{(kh, (x_{1}, \ldots, x_{n}))\;|\; k \in \mathrm{Aut}(C) \cong \PGL_{2}, (kh \circ \iota)(x_{i}) = p_{i}\},\] where we view $\PGL_2$ as the subgroup of~$\PGL_{d+1}$ fixing~$C$.  Moreover, since any~$n \ge 3$~points determine a unique automorphism of~$C$ that fixes them, $\psi^{-1}(\mathbf{p}) \cong \PGL_{2}$. Therefore \[\dim V_{d, n} = \dim \left(\PGL_{d+1} \times (\PP^{1})^{n}\right) - \dim \PGL_{2} = d^{2}+2d+n-3,\] as claimed. 
\end{proof}

Since it can be difficult to study a variety defined simply as a Zariski closure, we reinterpret $V_{d,n}$ in terms of a well-known (and well-behaved) moduli space.  Recall that $\overline{\rM}_{0,n}(\PP^d, k)$ is the coarse moduli space of the moduli stack~$\overline{\mathcal{M}}_{0,n}(\PP^{d}, k)$ of genus zero stable maps to~$\PP^d$ of degree~$k$. This is a projective variety (see~\cite{FP97} or~\cite[\S2.3]{KV07} for the characteristic zero case, \cite[Theorem 2.8]{AO01} for positive characteristic). It is also irreducible, normal, and of dimension~$dk + d + k + n - 3$ (see~\cite[Sections 3 and 4]{FP97} for the construction of a cover of $\overline{\rM}_{0,n}(\PP^d,k)$ by normal Zariski open subsets, which is valid for an arbitrary algebraically closed field~$\Bbbk$). A stable map is defined as follows:

\begin{definition}
Let~$X$ be a connected projective curve of arithmetic genus zero with at worst nodal singularities, and let~$x_{1}, \ldots, x_{n}$ be $n$~distinct smooth points of~$X$.  A morphism $f\colon X \to\PP^d$ is \emph{stable} if there are only finitely many automorphisms~$g \colon X \to X$ satisfying $g(x_{i}) = x_{i}$ and $f \circ g = f$. A stable map has \emph{degree}~$k$ if $f_{*}[X] = k[\ell]$ where $[\ell]$ is the homology class of a line in~$\PP^d$.
\end{definition}

The $n$ evaluation maps $\nu_{i} \colon \overline{\rM}_{0, n}(\PP^{d}, k) \to \PP^{d}$ send each equivalence class of stable maps \[(f \colon X \to \PP^d, x_{1},,\ldots,x_{n})\] to the point~$f(x_{i})$. These play a key role in Gromov-Witten theory, as it will appear later in the paper (see the proof of Proposition~\ref{prop:irreducible}).  By taking the product of all~$n$~evaluation maps we obtain the \emph{total evaluation map} \[\nu \colon \overline{\rM}_{0,n}(\PP^{d}, k) \to (\PP^{d})^{n},\] which will play a key role in our study, when~$k=d$.

\begin{proposition}\label{prop:basicprops}
Let $n \ge d+3$. 
\begin{enumerate}
\item The Veronese compactification~$V_{d, n}$ is the scheme-theoretic image of the total evaluation map~$\nu \colon \overline{\rM}_{0,n}(\PP^{d}, d) \to (\PP^{d})^{n}$.
\item For any point configuration $\mathbf{p} = (p_{i}) \in V_{d, n}$, there is a (possibly nodal) rational curve~$X$ and a stable map~$(f \colon X \to \PP^d, x_{1}, \ldots, x_{n})$ with~$f_{*}[X] = d[\ell]$ and $f(x_{i}) = p_{i}$. 
\end{enumerate}
\end{proposition}

\begin{proof}
Castelnuovo's Lemma implies that $\overline{\rM}_{0,n}(\PP^{d}, d)$ contains an open set parameterizing rational normal curves in~$\PP^d$ with $n$~distinct points.  Since $\overline{\rM}_{0,n}(\PP^{d}, d)$ is an irreducible proper variety, its image under~$\nu$ is then an irreducible closed subvariety of~$(\PP^d)^n$ containing~$U_{d, n}$ as a dense subset.  This implies that $\nu \colon \overline{\rM}_{0,n}(\PP^{d}, d) \to V_{d,n}$ is a surjective birational morphism with reduced scheme-theoretic image (see~\cite[Exercise II.3.11(d)]{Har77}).
\end{proof}

Some caution is needed when applying~(2) above: the image~$f(X)$ may not be a rational curve of degree~$d$ in~$\PP^{d}$, even when $X = \PP^{1}$, because $f$ may not be injective. For instance, $f$ could be a $d$-fold cover of a line.

To describe the boundary configurations in~$V_{d, n}$ and the curves supporting them, the following definition is useful.  (Recall that a curve embedded in projective space is said to be \emph{non-degenerate} if it is not contained in a hyperplane.)

\begin{definition}\label{def:quasiVeronese}
A \emph{quasi-Veronese curve} in~$\PP^{d}$ is a curve of degree~$d$ that is complete, connected, and non-degenerate.
\end{definition}

This terminology differs slightly from that of~\cite[Definition 2.1]{Gia13}, since here we require non-degeneracy.   The following result of Artin shows that quasi-Veronese curves are built out of rational normal curves in a fairly straightforward way.

\begin{lemma}[\protect{\cite[Lemma 13.1]{Art76}}]
Let $C$ be a quasi-Veronese curve in~$\PP^{d}$.  Then: 
\begin{enumerate}
\item each irreducible component of~$C$ is a rational normal curve in its span, 
\item the singularities of~$C$ are \'etale locally the union of coordinate axes in~$\Bbbk^{m}$, and
\item any connected closed subcurve of~$C$ is again a quasi-Veronese curve in its span.
\end{enumerate}
\end{lemma}

From Artin's description we see that the curves depicted in Figure~\ref{fig:quasiVeronese} are precisely the types of quasi-Veronese curves that occur in~$\PP^{3}$. 

\begin{figure}[!ht]
\centering
\includegraphics[scale=0.50,valign=t]{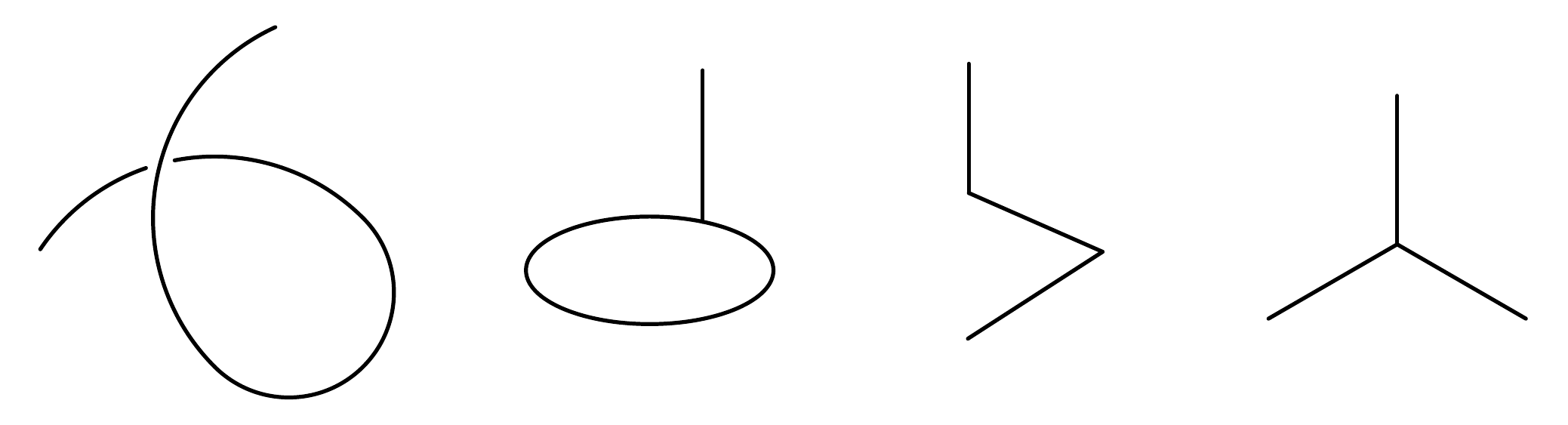}
\caption{The degree three quasi-Veronese curves: twisted cubic, non-coplanar union of line and conic, chain of three lines, and non-coplanar union of three lines meeting at a point.}
\label{fig:quasiVeronese}
\end{figure}

\begin{proposition}
\label{degeneratepointconfigonquasiVeronese}
A non-degenerate point configuration~$\mathbf{p} \in (\PP^{d})^{n}$ is in~$V_{d, n}$ if and only if it lies on a quasi-Veronese curve. 
\end{proposition}

\begin{proof}
Any non-degenerate image of a stable map in~$\overline{\rM}_{0,n}(\PP^{d}, d)$ is a quasi-Veronese curve. Therefore, if $\mathbf{p}\in V_{d,n}$, then it lies on a quasi-Veronese curve. Conversely, we show that quasi-Veronese curves are flat limits of rational normal curves. Consider the stack~$\overline{\mathcal{M}}_{0,n}(\mathbb{P}^d,d)$ and let $A\rightarrow\overline{\mathcal{M}}_{0,n}(\mathbb{P}^d,d)$ be an atlas, where $A$ is a scheme. Then there is a flat family $(\mathcal{A}\rightarrow~A, \mathcal{A} \to \PP^{d}, \sigma_{1}, \sigma_{2}, \cdots, \sigma_{n} \colon A \to \mathcal{A})$ consisting of all the object parametrized by~$\overline{\mathcal{M}}_{0,n}(\mathbb{P}^d,d)$. Denote by~$\mathcal{B}$ the image of~$\mathcal{A}\rightarrow A\times\mathbb{P}^d$, and consider the family~$\mathcal{B}\rightarrow A$ obtained by restricting the projection~$A\times\mathbb{P}^d\rightarrow A$. There exists a dense open subset~$A'\subseteq A$ such that the restricted family~$\mathcal{B}'\rightarrow A'$ parametrizes rational normal curves and quasi-Veronese curves of degree~$d$ in~$\mathbb{P}^d$. It follows from the irreducibility of~$\overline{\rM}_{0,n}(\mathbb{P}^d,d)$ that $A'$ is also irreducible. By induction on the number of irreducible components, one may check that every fiber of~$\mathcal{B}' \rightarrow A'$ has the same Hilbert polynomial. Thus $\mathcal{B}'\rightarrow A'$ is a flat family, showing that any quasi-Veronese curve is a flat limit of rational normal curves.
\end{proof}


\section{The case of conics}\label{sec:equationsofV2n}

The Veronese compactification for conics, $V_{2,n}$, turns out to be rather well-behaved and we are able to say a lot about it.  In this section we first find the single multi-homogeneous equation defining the hypersurface~$V_{2,6} \subseteq (\PP^2)^6$, see Proposition~\ref{prop:equationV26determined}, and then show that by pulling this equation back along all the forgetful maps we get defining equations for~$V_{2,n} \subseteq (\PP^2)^n$ for all~$n$, see Theorem~\ref{thm:eqnV2n}.  This shows, in particular, that the conic Veronese compactification~$V_{2,n}$ is a determinantal variety.  This helps us show that it is Cohen-Macaulay and normal for all~$n$, but Gorenstein only for~$n=6$.  We then turn to the question of finding subsets of our equations which cut out $V_{2,n}$ set-theoretically.  We show that this question is answered precisely by a property of hypergraphs that has been studied by many authors from various combinatorial perspectives, see Theorem~\ref{thm:V2nT}.  

\subsection{The unique equation defining $V_{2,6}$}\label{ssec:equationofV26}

Five general points determine a unique conic and it is a codimension one condition for a sixth point to lie on this conic, so $V_{2, 6} \subseteq (\PP^{2})^{6}$ is a hypersurface and hence is defined by a single multi-homogeneous equation.  This defining equation is classical and well-known (see~\cite[Example~3.4.3]{Stu08}).  In this subsection we will recall one method for producing it, since both the equation and the construction will be useful later on.   In this subsection we denote the homogeneous coordinates of~$(\PP^2)^6$ by \[a=[a_0:a_1:a_2],~b=[b_{0}:b_{1}:b_{2}],~\ldots,~f=[f_0:f_1:f_2].\]

 \begin{definition}
Consider the following multi-homogeneous polynomial of degree~$(2,\ldots,2)$:
\[\label{equation-V26}
\phi(a,\ldots,f):=\left| \begin{array}{cccccc}
a_0^2&b_0^2&c_0^2&d_0^2&e_0^2&f_0^2\\
a_1^2&b_1^2&c_1^2&d_1^2&e_1^2&f_1^2\\
a_2^2&b_2^2&c_2^2&d_2^2&e_2^2&f_2^2\\
a_0a_1&b_0b_1&c_0c_1&d_0d_1&e_0e_1&f_0f_1\\
a_0a_2&b_0b_2&c_0c_2&d_0d_2&e_0e_2&f_0f_2\\
a_1a_2&b_1b_2&c_1c_2&d_1d_2&e_1e_2&f_1f_2
\end{array} \right |,\]
and let~$W_{2,6} = Z(\phi(a,\ldots,f)) \subseteq (\PP^2)^6$ denote the corresponding hypersurface.
\end{definition}

The following result says that $\phi(a,\ldots,f)$ is the unique equation defining $V_{2,6}$. 

\begin{proposition}\label{prop:equationV26determined}
As schemes, we have $V_{2,6} = W_{2,6}$.
\end{proposition}

\begin{proof}
Let~$v \colon \PP^2\hookrightarrow\PP^5$ be the Veronese embedding
\begin{equation*}
	[z_0:z_1:z_2]\mapsto[z_0^2:z_1^2:z_2^2:z_0z_1:z_0z_2:z_1z_2].
\end{equation*}
Any conic in~$\PP^{2}$ is the restriction to~$v(\PP^{2})$ of a hyperplane in~$\PP^{5}$. Thus six points $a,\ldots,f\in\PP^2$ lie on a conic if and only if $v(a),\ldots, v(f)$ lie on a hyperplane, which is true if and only if $\phi(a,\ldots,f)=0$. This shows that $V_{2, 6} = W_{2, 6}$ set-theoretically.  Since $V_{2,6}$ is reduced, to show that this equality holds scheme-theoretically it suffices to show that $\phi$ is not a square.

Assume by contradiction that $\phi(a,\ldots,f)=p^2$ for some multihomogeneous polynomial~$p$ of degree~$(1,\dots,1)$.
	In particular, we can write $p^2=(\lambda a_0+\mu a_1+\nu a_2)^2$, with $\lambda,\mu,\nu$ polynomials in the variables~$b,c,d,e,f$.
	By expanding the determinant along the first row, we obtain that $\lambda^2$ is equal to the~$5\times 5$~minor obtained from~ $\phi(a,\ldots,f)$ by removing the first row and column.
	Repeating the same argument for the second and the third row, we obtain that the~$3\times3$~minor formed by the last~$3$~rows and columns of $\phi(a,\ldots,f)$ must be a square as well. 
	On the other hand, this is the determinant of a $3\times3$ matrix whose entries are algebraically independent, therefore it is irreducible by~\cite[Theorem~2.10]{BV88}, and we obtain a contradiction.
\end{proof}

\begin{remark}\label{rem:bracket}
Since $\phi(a, \ldots, f)$ is $\SL_{3}$-invariant, by the fundamental theorem of invariant theory it is a polynomial in the maximal minors of coordinates on~$(\PP^2)^6$.  Indeed, 
\[
	\phi(a, \ldots, f) = |abc||ade||bdf||cef|-|abd||ace||bcf||def|,
\]
where, e.g., $|abc|$ denotes the determinant of the matrix whose columns are the coordinates of the points~$a$, $b$, and $c$.  (Cf.,~(\cite[p.118]{Cob61} and \cite[Example~3.4.3]{Stu08}.)
\end{remark}


\subsection{Pulling back from $V_{2,6}$ to $V_{2,n}$}\label{ssec:equationsV2n}
We will show here that the equations for~$V_{2,n}$ are all obtained by pulling back the unique equation for~$V_{2,6}$ along the natural forgetful maps.  

For any subset~$I \subseteq [n]$, there is a projection map \[\pi_{I} \colon (\PP^{d})^{n} \to (\PP^{d})^{|I|},~(p_{i}) \mapsto (p_{i})_{i \in I}.\] This restricts to a surjective map $V_{d, n} \twoheadrightarrow V_{d, |I|}$ that we also denote by~$\pi_{I}$. 

\begin{definition}\label{def:W2n}
Let~$\phi$ be the multi-homogeneous form defining $V_{2, 6} \subseteq (\PP^{2})^{6}$ (cf.,~\S\ref{ssec:equationofV26}).
\begin{enumerate}
\item For any subset~$I \subseteq [n]$ with~$|I| = 6$, let~$\phi_{I} := \pi_{I}^{*}(\phi)$. \item For~$n\ge 6$, let~$W_{2, n} \subseteq (\PP^{2})^{n}$ be the closed subscheme defined by~$\phi_{I}$ for all~$I \subseteq [n]$ with~$|I| = 6$.
\end{enumerate}
\end{definition}

Clearly~$V_{2, n}\subseteq W_{2,n}$, since if $n$ points in~$\PP^2$ lie on a conic then so do any 6 of them, and we have shown above that $V_{2,6} = W_{2,6}$.  A very simple but useful observation is the following:

\begin{observation}\label{obs:fibersofV}
For~$n\ge 7$, the fibers of the projection map $\pi_{[n-1]} \colon  V_{2,n}\to V_{2,n-1}$ forgetting the last coordinate are, set-theoretically, either a smooth conic, two distinct lines, or~$\PP^2$. A single line cannot occur as a fiber because if $(p_{1}, \ldots, p_{n-1}) \in V_{2, n-1}$ lie on a line, then for any point~$p_{n} \in \PP^{2}$, $(p_{1}, \ldots, p_{n})$ lie on a nodal conic, and so they determine a point in~$V_{2,n}$. 
\end{observation}

\begin{theorem}\label{thm:eqnV2n}
The Veronese compactification~$V_{2, n}$ is defined by~$\phi_{I}$ for all~$I \subseteq [n]$ with $|I|=6$, i.e., $V_{2, n} = W_{2, n}$ as schemes. 
\end{theorem}

\begin{proof}
First, we show by induction on~$n$ that the equality $V_{2, n} = W_{2, n}$ holds set-theoretically. The base case~$n = 6$ is Proposition \ref{prop:equationV26determined}. Fix~$n \ge 7$, assume by the inductive hypothesis that  $V_{2, n-1}= W_{2, n-1}$ set-theoretically, and consider the following diagram:
\[
	\xymatrix{V_{2,n} \ar@{^{(}->}[r] \ar@{->>}[d]^{\pi_{[n-1]}} & 
	W_{2, n} \ar@{^{(}->}[r] \ar@{->>}[d]^{\pi_{[n-1]}} &
	(\PP^{2})^{n} \ar@{->>}[d]^{\pi_{[n-1]}}\\
	V_{2, n-1} \ar@{=}[r] & W_{2,n-1} \ar@{^{(}->}[r] & (\PP^{2})^{n-1}.}
\]
For any point $\mathbf{p}=(p_{1},\ldots,p_{n-1})\in V_{2, n-1}=W_{2, n-1}$, consider the fiber $V_{\mathbf{p}}$ (resp.~$W_{\mathbf{p}}$) of $V_{2, n} \rightarrow V_{2, n-1}$ (resp.~$W_{2, n} \rightarrow W_{2, n-1}$). Clearly $V_{\mathbf{p}} \subseteq W_{\mathbf{p}} \subseteq \PP^2$ and we are done if we show that the first containment is always an equality.  If~$V_{\mathbf{p}}=\PP^2$ then equality is automatic, also as schemes, so by~Observation~\ref{obs:fibersofV} we may assume that $V_{\mathbf{p}}$ is either a smooth conic or two distinct lines.   By the definition of~$W_{2,n}$, the fiber is given by
\begin{equation}
\label{Wasintersection}
W_{\mathbf{p}}=\bigcap_{\substack{n\in I\subseteq [n],\\|I|=6}}Z(\phi_I(\mathbf{p},[z_0:z_1:z_2])),
\end{equation}
where the~$z_i$ are homogeneous coordinates on our~$n$-th copy of~$\PP^2$ and each $\phi_I(\mathbf{p},[z_0:z_1:z_2])$ is either $0$ or a nonzero homogeneous polynomial of degree~$2$ in the~$z_i$. 
We are done if there exists $J\subseteq[n]$ with~$n\in J$ and~$|J|=6$ such that the polynomial $\phi_J(\mathbf{p},[z_0:z_1:z_2])$ is nonzero. 
In fact, since $\phi_J(\mathbf{p},[z_0:z_1:z_2])$ has degree~$2$ it follows that $W_{\mathbf{p}}$ is contained in a conic, and so it must be equal to~$V_{\mathbf{p}}$.
In particular, since $V_{\mathbf{p}}$ is not a line, the polynomial $\phi_J(\mathbf{p},[z_0:z_1:z_2])$ cannot be a square of a polynomial of degree~$1$, so it is reduced. 
This must happen for each nonzero polynomial of~\eqref{Wasintersection}, therefore we obtain the equality $V_{\mathbf{p}}= W_{\mathbf{p}}$ as schemes, i.e. geometric fibers are reduced.
 
To show this, let us assume by contradiction that there is no nonzero $\phi_J(\mathbf{p},[z_0:z_1:z_2])$, and hence that $W_{\mathbf{p}}=\PP^2$. This implies that any 5 points from $\mathbf{p}$ do not determine a unique conic, even though all~$n-1$ of them do since by assumption $V_\mathbf{p}$ is a smooth conic or two distinct lines. Each point $p_{i}$ defines a linear functional $\ell_{i} \in \mathrm{H}^{0}(\mathcal{O}_{\PP^{2}}(2))^{*}$, and the intersection $\bigcap_{i=1}^{n-1}\mathrm{ker}\;\ell_{i}$, which is one-dimensional, is the set of conics passing through all of the points $p_{i}$. There must be five linearly independent functionals $\{\ell_{i_j}\}_{j=1}^5$ such that $\bigcap_{j=1}^5\mathrm{ker}\;\ell_{i_j}$ is one-dimensional, but then setting $J=\{i_1,\ldots,i_5,n\}$ gives $\phi_J(\mathbf{p},[z_0:z_1:z_2]) \ne 0$, a contradiction. Thus $V_{2, n} = W_{2, n}$ as sets.

Next, we argue by induction on~$n$ that $W_{2, n}$ is reduced, and hence that $V_{2, n} = W_{2, n}$ as schemes. Again, the base case $n=6$ is Proposition \ref{prop:equationV26determined}.  Since being reduced is a local property, we may replace $W_{2, n}$ and $W_{2, n-1} = V_{2, n-1}$ by affine neighborhoods $\spec \;B$ and $\spec\;A$, respectively. Thus $A$ is a reduced ring by inductive hypothesis, the projection map corresponds to an injective ring homomorphism $A \hookrightarrow B$, and our goal is to show that $B$ is reduced. Suppose that there is $b \in B$ such that $b^{t} = 0$. We may assume that $b \notin A$, because $A$ is reduced. Let $m$ be a maximal ideal of~$A$. Since every geometric fiber is reduced, $A/m\otimes_{A}(b) =0$. For any maximal ideal~$n$ of~$B$ such that $n \cap A = m$, we have \[(b)/n(b) \cong A/m\otimes_{A}(b)/n(b) \cong A/m\otimes_{A}((b) \otimes_{B}B/n) \cong (A/m \otimes_{A}(b)) \otimes_{B}B/n = 0.\] (To understand the first isomorphism of the chain, let $M=(b)/n(b)$. Then $mM=0$, so that $M=M/mM\cong A/m\otimes_AM$.) By Nakayama's lemma, $b = 0$ for some open neighborhood of the closed point $n$. Since this is true for all closed points, $b = 0$ in $\spec \;B$. 
\end{proof}

\subsection{Geometric properties}\label{ssec:propertiesV2n}

One significant consequence of the equational description of~$V_{2.n}$ established in Theorem~\ref{thm:eqnV2n} is that this $d=2$~Veronese compactification is a determinantal variety in the sense of \cite[\S18]{Eis95}.  Indeed, locally on affine charts $V_{2, n}$ is defined by an ideal generated by~$6 \times 6$~minors of a~$6 \times n$~matrix with entries in $\Bbbk[X_{i}, Y_{i}]_{1 \le i \le n}$, where $(X_i,Y_i)$ are affine coordinates for each copy of~$\mathbb{P}^2$.

\begin{corollary}\label{V2nisCM}
The Veronese compactification~$V_{2, n}$ is Cohen-Macaulay. 
\end{corollary}

\begin{proof}
The defining ideal on affine charts has codimension~$2n - (n+5) = n-5$, so this follows from~\cite[Theorem~18.18]{Eis95}.
\end{proof}

The Cohen-Macaulay property, together with a careful analysis of the relation between $V_{2,n}$ and $\overline{\rM}_{0,n}(\PP^{2},2)$ on a sufficiently nice locus, allows us to obtain another important geometric property:

\begin{theorem}\label{normalityofV2n}
The Veronese compactification~$V_{2, n}$ is normal. 
\end{theorem}

\begin{proof}
This is trivial for~$n \le 5$, so assume $n \ge 6$.  Since Serre's condition asserts that normality is equivalent to properties~$R_1$ and~$S_2$ holding, and Cohen-Macaulay is equivalent to all $S_k$ holding, it suffices to prove that $V_{2,n}$ is regular in codimension 1.  Let~$\rN' \subseteq \overline{\rM}_{0,n}(\PP^{2},2)$ be the union of the closures of the following three loci of stable maps $(f \colon X \to \PP^{2}, x_{1}, \ldots, x_{n})$: (1) the image of $f$ is a line; (2) the domain is reducible, $X = X_{1} \cup X_{2}$, with $X_1$ mapped to a point and containing at least three marked points; (3) the domain is reducible, $X = X_{1} \cup X_{2}$, with only one marked point on $X_{1}$ and $\deg f_{*}[X_{1}] = 1$.  Let~$\rN$ be the union of~$\rN'$ with the pre-image along the total evaluation map~$\nu \colon \overline{\rM}_{0,n}(\PP^{2},2) \to (\PP^2)^n$ of the closed locus parameterizing point configurations with no five points in general linear position. It is straightforward to see that (i) the restriction of the total evaluation map $\nu \colon \rN \to \nu(\rN)$ is a positive-dimensional fibration while its restriction to the open complement $\rM := \overline{\rM}_{0, n}(\PP^{2}, 2) \setminus \rN$ is bijective, and (ii) if we set $V := \nu(\rM)$ then the codimension of $V_{2, n} \setminus V$ is at least two.  Thus, if we show that $V$ is $R_1$, then it follows that $V_{2,n}$ is $R_1$. To show regularity in codimension 1 of $V$ it suffices to show that the bijective morphism $\nu \colon \rM \to V$ is an isomorphism, because $\rM \subseteq \overline{\rM}_{0, n}(\PP^{2}, 2)$ is normal. To do this, we explicitly construct the inverse morphism. 

For the projection $\pi_{[n]} \colon V_{2, n+1} \to V_{2, n}$, set $V_{+} := \pi_{[n]}^{-1}(V)$ and consider the restriction $\pi_{[n]} \colon V_{+} \to V$. This is a flat family since the fibers are all non-degenerate conics, hence have the same Hilbert polynomial, and the base is reduced. (Note that the fiber over $\mathbf{p}\in V$ cannot be $\mathbb{P}^2$ because there are at least five points $p_{i_1},\ldots,p_{i_5}$ in general linear position.) Furthermore, there are $n$ sections $\sigma_{1}, \ldots, \sigma_{n} \colon V \to V_{+}$ given by 
\[
	\sigma_{i}(p_{1}, \ldots, p_{n}) = (p_{1}, \ldots, p_{n}
	, p_{i}).
\]
For two distinct indices $1 \le i , j \le n$, let $\Delta_{i,j} \subseteq V_{+}$ be the smooth codimension two subvariety defined by $p_{n+1} = p_{i} = p_{j}$. Take the blow-up $\widetilde{V}_{+} \to V_{+}$ of all the $\Delta_{i,j}$ (the order doesn't matter since they are disjoint) and let $\widetilde{\sigma}_{i} \colon V \to \widetilde{V}_{+}$ be the proper transform of the above sections. Consider the family $\widetilde{V}_{+} \to V$ together with the sections $\widetilde{\sigma}_{1}, \ldots, \widetilde{\sigma}_{n}$ and the morphism given by the composition $\widetilde{V}_{+} \to V_{+} \xrightarrow{\pi_{n+1}} \PP^{2}$. Then $\widetilde{V}_{+} \to V$ is a flat family of stable maps inducing a morphism $V \to \rM$ that one can verify is the inverse of~$\nu$. 
\end{proof}

A natural question which may arise at this point is whether the variety~$V_{2,n}$ is also Gorenstein.  The answer is the following:

\begin{proposition}
The Veronese compactification~$V_{2, n}$ is Gorenstein if and only if~$n=6$. 
\end{proposition}
\begin{proof}
This can be checked locally.
The stalk at a point of~$V_{2, n}$ is given by~$A=R/I$, where $R$ is a local regular ring of dimension~$2n$ and $I$ is an ideal of codimension~$n-5$ generated by~$6 \times 6$~minors of a~$6 \times n$~matrix~$M$ with entries in~$R$.
By Corollary~\ref{V2nisCM}, $A$ is Cohen-Macaulay of dimension~$n+5$, therefore by~\cite[Theorem~3.2.10]{BH98} $A$ is Gorenstein if and only if $\mathrm{Ext}_A^{n+5}(\Bbbk,A)$ has $\Bbbk$-dimension~$1$.
We have an isomorphism \[\mathrm{Ext}_A^{n+5}(\Bbbk,A)\cong\mathrm{Tor}^R_{2n-n-5}(\Bbbk,A),\] therefore $A$ being Gorenstein is equivalent to the last non-zero Betti number \[\beta_{n-5}^R(A)=\dim_{\Bbbk}\mathrm{Tor}^R_{n-5}(\Bbbk,A)\] being equal to $1$.
Since $\mathrm{codim}\; I = n-5$, by~\cite[Theorem~2.16]{BV88} a $R$-free resolution of~$R/I$ is given by the Eagon-Northcott complex of the map $\varphi_M\colon R^n\rightarrow R^6$ corresponding to the matrix~$M$.
The last non-zero module of the complex is in homological position~$n-5$ and it is~$\mathrm{Sym}^{n-6}(R^6)$. Now, $A$ is Gorenstein if and only if this free module has rank~$1$, which holds if and only if~$n=6$.	
\end{proof}

\subsection{Toward a smaller set of defining equations}\label{sec:smallereqs}

We showed above that the ideal defining $V_{2,n}$ is generated by~$\binom{n}{6}$~polynomials, namely, the pull-backs of the defining polynomial for~$V_{2,6}$ along all forgetful maps $V_{2,n} \to V_{2,6}$.  Here we will show that $V_{2,n}$ can be cut out set-theoretically by certain proper subsets of these polynomials.  We first study the analogous problem of minimal collections of set-theoretic equations for generic determinantal varieties in a product of projective spaces and then show how the case of~$V_{2,n}$, which is determinantal but with relations among the matrix entries, reduces to this case of indeterminate entries.  In both cases, the combinatorics controlling which subsets of polynomials are set-theoretically valid turns out to be an interesting problem concerning hypergraphs and set transversals that has been studied from multiple perspectives in the combinatorics literature.

\begin{definition}\label{def:Ydn}
The \emph{space of degenerate point configurations} (those lying on a hyperplane) is the determinantal variety~$Y_{d, n} \subseteq (\PP^{d})^{n}$ defined by all $(d+1) \times (d+1)$~minors of the $(d+1)\times n$~matrix whose columns are given by the homogeneous coordinates of each copy of~$\PP^d$.
\end{definition}

Note that $Y_{d, n} = (\PP^{d})^{n}$ if~$n \le d$, since in this case there are no size~$d+1$~minors (and also any $d$~points in~$\PP^d$ lie on a hyperplane), so we are generally interested in the case~$n \ge d+1$.  

For a subset $H \subseteq [n]$ with~$|H| = d+1$, let $m_{H}$ be the minor in our matrix of homogeneous coordinates corresponding to the columns indexed by~$H$, and for a collection $\cT \subseteq {[n] \choose d+1}$ of such subsets let~$Y_{d, n}^{\cT} \subseteq (\PP^{d})^{n}$ be the closed subscheme defined by the multi-homogeneous polynomials $\{m_{H}\}_{H \in \cT}$. By definition we have $Y_{d, n}^{{[n] \choose d+1}} = Y_{d, n}$ and if $\cT_{1} \subseteq \cT_{2}$ then~$Y_{d, n}^{\cT_{2}} \subseteq Y_{d, n}^{\cT_{1}}$. 

\begin{question}
Two natural questions immediately arise:
\begin{enumerate}
\item For which collections~$\cT$ does the equality $Y_{d, n}^{\cT} = Y_{d, n}$ hold?
\item What is the minimal cardinality of a collection~$\cT$ such that $Y_{d, n}^{\cT} = Y_{d, n}$?
\end{enumerate}
\end{question}

Both questions can be asked at the level of varieties (that is, set-theoretically) or at the level of schemes.  We shall focus on the former, since in that case we find a very natural combinatorial answer.

Recall that a \emph{$k$-uniform hypergraph}~$\mathcal{H}$ is by definition a subset of~${[n] \choose k}$, and each $H\in \mathcal{H}$ is called an \emph{edge}. 

\begin{definition}
We say that the \emph{transversality property} holds for a $k$-uniform hypergraph~$\mathcal{H}$  (or that $\mathcal{H}$ satisfies the transversality property) if for any partition~$I_{1} \sqcup \cdots \sqcup I_{k} = [n]$ with each~$I_j$ nonempty, there is an edge~$H \in \mathcal{H}$ such that $|H \cap I_{j}| = 1$ for all~$1 \le j \le k$. 
\end{definition}

Note that a $2$-uniform hypergraph is just a graph in the usual sense, and the transversality property is equivalent to connectedness of the graph.  See~\cite[\S1]{BT09} for a brief survey of the literature surrounding this concept of hypergraph transversality, including translations to other combinatorial settings.

\begin{example}
For~$n=5$, consider the $3$-uniform hypergraph
\[
\mathcal{H}=\{\{1,2,3\},\{2,3,4\},\{3,4,5\},\{4,5,1\},\{5,1,2\}\}.
\]
It is simple to verify that $\mathcal{H}$ satisfies the transversality property. Moreover, there does not exist a $3$-uniform hypergraph on~$[5]$ with fewer than five edges that satisfies the transversality property.  Up to permutation (that is, the $S_5$-action on~$[5]$), $\mathcal{H}$ is the unique hypergraph satisfying the transversality property with the minimal number of edges for~$(k,n)=(3,5)$.
\end{example}

The next result answers the first of the two question stated above.

\begin{proposition}\label{prop:YdnT}
For a collection~$\cT \subseteq {[n] \choose d+1}$, we have a set-theoretic equality $Y_{d, n}^{\cT} = Y_{d, n}$ if and only if the transversality property holds for~$\cT$, viewed as a hypergraph. 
\end{proposition}

\begin{proof}
We may regard each point in~$(\PP^{d})^{n} = (\mathbb{A}^{d+1}\setminus\{\mathbf{0}\})^n/(\Bbbk^\times)^n$ as an equivalence class --- under the action of multiplication on the left by diagonal invertible matrices --- of~$(d+1) \times n$~matrices none of whose columns is the zero vector. Suppose that $Y_{d, n}^{\cT} = Y_{d, n}$. For any partition $I_{1} \sqcup \cdots \sqcup I_{d+1} = [n]$, consider the matrix $A \in (\PP^{d})^{n}$ whose $i$-th column is the $j$-th standard basis vector $e_{j}$ when $i \in I_{j}$. This matrix has full rank, so $A \notin Y_{d,n} = Y_{d,n}^{\cT}$, and hence there must be a nonzero minor $m_{H}(A) \ne 0$ with $H \in \cT$. By our construction of~$A$, a nonzero minor cannot have more than one column indexed by any~$I_{j}$, so we have $|H \cap I_{j}| = 1$ for all~$j$ and thus $\cT$ satisfies the transversality property. 

Conversely, suppose $\cT$ has the transversality property and let~$B$ be a matrix representing an arbitrary point of~$ (\PP^d)^n\setminus Y_{d, n}$, i.e., $B$ is a full rank $d+1$ matrix without zero columns.  If we can show that $B \notin Y_{d,n}^{\cT}$ then we will obtain $Y_{d, n}^{\cT} \subseteq Y_{d, n}$, and hence $Y_{d, n}^{\cT} = Y_{d, n}$ since the opposite containment is trivial.  First, note that both $Y_{d,n}$ and $Y_{d, n}^{\cT}$ are $\GL_{d+1}$-invariant subsets of~$(\PP^d)^n$, for the natural diagonal action, so we can assume without loss of generality that $B$ is in row echelon form.  For each~$1 \le j \le d+1$, let~$I_{j}$ be the set of columns such that the lowest nonzero entry is on the $j$-th row. Then $I_{1} \sqcup \cdots \sqcup I_{d+1}$ is a partition of~$[n]$, with~$I_{j} \ne \emptyset$ for each~$j$, since $\mathrm{rank}(B) = d+1$. By assumption on~$\cT$ there must exist $H \in \cT$ such that $|H \cap I_{j}| = 1$ for all~$j$.  But then clearly $m_{H}(B) \ne 0$, and so $B \notin Y_{d, n}^{\cT}$ as desired.
\end{proof}

By the preceding result, the most ``efficient'' ways to cut out the locus of degenerate point configurations correspond to the uniform hypergraphs~$\cT$ with the minimal number of edges that satisfy the transversality property.  We can derive a lower bound for this minimal number of edges as follows.  Consider a partition of~$[n]$ where $|I_{j}| = 1$ for $j \le d$ and $|I_{d+1}| = n-d$. The set of such partitions can be identified with~${[n] \choose n-d}$. Fix a hypergraph~$\cT$ with the transversality property. Consider the following incidence locus: 
\[Inc := \{(J, H) \in {[n] \choose n-d} \times \cT \;|\; |J \cap H| = 1\}.\]
We have two projections 
\[
	\xymatrix{& Inc \ar[ld]_{\pi_{1}} \ar[rd]^{\pi_{2}}\\
	{[n] \choose n-d} && \cT.}
\]
The transversality property of~$\cT$ implies that $\pi_{1}$ is surjective, so $|Inc| \ge {n \choose n-d} = {n \choose d}$. For any~$H \in \cT$, we have $|\pi_{2}^{-1}(H)| = d+1$. Thus, 
\[
	|\cT| = \frac{|Inc|}{d+1} \ge \frac{1}{d+1}{n \choose d}.
\]
This is not far from the best known lower bound~\cite{Ste75} (see also~\cite[\S1.4]{BT09}): 
\[|\cT| \ge \frac{2}{n-d+2}\binom{n}{d}.\]

We now return to the question of finding a smaller set of defining equations for~$V_{2, n}$.  Remarkably, the following result says that at least set-theoretically, the combinatorics of which determinantal equations can be dropped when defining $V_{2,n}$ is exactly the same as that of the space of degenerate point configurations in~$\PP^5$, namely~$Y_{5,n}$, analyzed above.

\begin{theorem}\label{thm:V2nT}
For a collection~$\cT \subseteq {[n] \choose 6}$, the equations $\{ \phi_{H}\}_{H \in \cT}$ define $V_{2, n}$ set-theoretically if and only if $\cT$ satisfies the transversality property. 
\end{theorem}

\begin{proof}
Let~$V_{2, n}^{\cT} \subseteq (\PP^{2})^{n}$ be the subscheme defined by~$\{\phi_{H}\}_{H \in \cT}$. For the product of Veronese embeddings $v \colon (\PP^{2})^{n} \hookrightarrow (\PP^{5})^{n}$, we have $v(V_{2, n}^{\cT}) = \mathrm{im}\; v \cap Y_{5, n}^{\cT}$ (cf., the proof of Proposition~\ref{prop:equationV26determined}). 

If $\cT$ satisfies the transversality property, then by Proposition~\ref{prop:YdnT} we have $Y_{5, n}^{\cT} = Y_{5, n}^{\binom{[n]}{6}}$, so intersecting with~$\mathrm{im}\; v$ yields $V_{2, n}^{\cT} = V_{2, n}^{\binom{[n]}{6}}$.  But~$V_{2,n}^{\binom{[n]}{6}} = V_{2,n}$ by Theorem~\ref{thm:eqnV2n}, so~$V_{2, n}^{\cT} = V_{2, n}$.  Conversely, suppose $V_{2, n}^{\cT} = V_{2, n}$. In the first half of the proof of Proposition~\ref{prop:YdnT}, we used a matrix whose columns were repetitions of the $d+1$~standard basis vectors. The same argument works for repetitions of any $d+1$~linearly independent vectors. Because $v(\PP^{2}) \subseteq \PP^{5}$ is non-degenerate, we can find $6$~linearly independent vectors in~$v(\PP^{2})$. Thus the same proof works here to show that the transversality property holds for~$\cT$. 
\end{proof}

\begin{example}\label{ex:V27eqs}
There is only one equation for~$V_{2,6}$.  For~$V_{2,7}$, our construction in Theorem~\ref{thm:eqnV2n} uses $\binom{7}{6} = 7$ equations.  We can drop any one of these and the corresponding hypergraph still satisfies the transversality property, so by the preceding result any 6 of these 7 equations cut out $V_{2,7}$ set-theoretically. 
\end{example}

\begin{remark}
A reasonable expectation is that if the equality $V_{2, n}^{\cT} = V_{2, n}$ holds set-theoretically, then it also does scheme-theoretically.  We verified this using a computer algebra system in a few cases for small values of~$n$, but more evidence would be necessary to confidently state the conjecture. 
\end{remark}

\begin{remark}
The \emph{arithmetical rank} of a polarized projective variety~$X \subseteq \PP^{N}$ is the smallest number of equations needed to define $X$ set-theoretically.  The definition makes sense for varieties embedded in products of projective space as well.  Bruns and Schw\"anzl proved that the arithmetical rank of the ideal of $t$~minors of an~$m \times n$~generic matrix is exactly $mn - t^{2}+1$~\cite[Theorem 1]{BS90}.  In particular, Bruns gives an explicit construction of a set of defining equations by using a poset attached to the matrix~\cite[Corollary 2.2]{Bru89}.  We can apply his construction to the set of equations $\{\phi_I\}$  defining $V_{2, n} \subseteq (\PP^{2})^{n}$, obtaining a new set of equations, certain sums of the~$\phi_I$, that set-theoretically defines $V_{2, n} \subseteq (\PP^{2})^{n}$.  In this way we obtain an upper bound of~$6n - 35$ on the arithmetical rank of~$V_{2, n}$.  Contrary to the case of a generic matrix  this upper bound is not sharp.  For example, for~$n=7$ it gives $6\cdot7-35=7$ equations, whereas we saw in Example~\ref{ex:V27eqs} that $V_{2,7}$ can be defined set-theoretically by~$6$~equations.  Nonetheless, since this upper-bound is linear in~$n$, it is asymptotically much better than the one obtained by combining Theorem~\ref{thm:V2nT} with the combinatorial bounds in~\cite{BT09}.
\end{remark}


\section{The Veronese compactification in higher dimensions}

In this section we turn to the Veronese compactification~$V_{d,n} \subseteq (\PP^d)^n$ for~$d > 2$.  Our main tool for producing equations for this variety is the classical Gale transform, which allows us to draw from the~$d=2$~case studied in the previous section.  However, the Gale transform does not apply to degenerate point configurations, and this leads to a serious complication.  

The equations we produce cut out (conjecturally, though we prove some important special cases) a union of two irreducible loci: $V_{d,n}$ and the locus $Y_{d,n}$ of degenerate configurations, which made a brief appearance in~\S\ref{sec:smallereqs}.  These two components coincide only in a few small cases; in general, finding more equations to cut out exactly $V_{d,n}$ appears to be very challenging, and we note multiple reasons why.  

Since the Gale transform appears in various forms in the literature, and since we shall need some very specific properties of it that are difficult to find in the literature, we begin this section with a study of the Gale transform.  (For a generalization, historical background, and more scheme-theoretic investigation see~\cite{EP00}.)

\subsection{Gale duality}\label{sec:Gale}

We denote the set of~$a\times b$~matrices by~$M_{a\times b}$ and the subset of full rank matrices by~$M_{a\times b}^{full}$.  If~$a < b$, then for a matrix~$A\in M_{a\times b}$ and an index set~$I\in\binom{[b]}{a}$  we let~$m_I(A)$ denote the associated maximal minor of~$A$.  Throughout, assume~$2 \le d \le n-2$.  

\begin{definition}\label{def:affineGaletransform}
For~$A \in M^{full}_{(d+1) \times n}$, the set of \emph{affine Gale transforms} of~$A$ is
\[\widetilde{G}(A) := \{B\in M^{full}_{(n-d-1) \times n}~|~AB^{t} = 0\}.\]
\end{definition}

\begin{remark}\label{rem:GaleOrbits}
We shall use repeatedly the following elementary observations: 
\begin{enumerate}
\item $\GL_{n-d-1}$ acts transitively on~$\widetilde{G}(A)$, and
\item $\widetilde{G}(\cdot)$ is constant on~$\GL_{d+1}$-orbits. 
\end{enumerate}
\end{remark}

\begin{example}\label{ex:affineGale}
Assume the first $d+1$ columns of $A \in M^{full}_{(d+1) \times n}$ are linearly independent. Then the $\GL_{d+1}$-orbit of $A$ contains a matrix of the form $[\mathrm{Id}_{d+1} | \underline{A}]$ for some $\underline{A} \in M_{(d+1) \times (n-d-1)}$, where for a positive integer~$m$ we use~$\mathrm{Id}_m$ to denote the~$m\times m$~identity matrix. It is then straightforward to check that $[\underline{A}^{t} | - \mathrm{Id}_{n-d-1}] \in \widetilde{G}(A)$.
\end{example}

The Gale transform exhibits a compatibility between maximal minors, see Proposition~\ref{prop:Galetransformminor}, that we shall rely upon heavily when studying configurations on rational normal curves.  First, some notation:

\begin{definition}\label{def:S}
For~$I = \{i_1,\ldots,i_k\} \in \binom{[n]}{k}$, let 
\[S_I := \sum_{j=1}^{k}(i_{j}-j)\] be the minimum number of adjacent transpositions $(i, i+1)$ needed to obtain $I$ from~$[k] \subseteq[n]$.  
\end{definition}

The following notation is also helpful, though we only use it in the immediately following proposition and its proof.   For~$I \subseteq [n]$, let $I_{1} := I \cap [d+1]$ and $I_{2} := I \setminus I_{1}$, so $I = I_{1}\sqcup I_{2}$; let $J_{1} := [d+1] \setminus I_{1}$ and $J_{2} := I^{c} \setminus [d+1]$, so $I^{c} = J_{1} \sqcup J_{2}$.  If $A\in M_{a\times b}$ and $I\subseteq [b]$, $J\subseteq [a]$ satisfy $|I| = |J|$, then we let $m_{I}^{J}(A)$ be the minor of $A$ specified by the $I$-columns and $J$-rows of $A$; if this is a maximal minor, say $J=[a]$, we denote this simply by~$m_I(A)$. 

\begin{proposition}\label{prop:Galetransformminor}
Let~$A \in M_{(d+1) \times n}^{full}$, $B \in \widetilde{G}(A)$, and $I\in\binom{[n]}{d+1}$.  Then 
\[
	m_{I}(A) = (-1)^{S_{I}+(n-d-1)}\lambda m_{I^{c}}(B),
\]
where $\lambda \in \Bbbk^\times$ is a nonzero constant independent of~$I$. 
\end{proposition}

\begin{proof}
By continuity it suffices to prove the statement for a general matrix $A$, so we may assume that all the minors of~$A$ are nonzero. Let us start by proving our identity in the special case $A = [\mathrm{Id}_{d+1}|\underline{A}]$ and $B = [\underline{A}^{t}|-\mathrm{Id}_{n-d-1}]$. 

The columns of~$\underline{A}$ are indexed by~$d+2, d+3, \ldots, n$. By the cofactor expansion, 
\[
	m_{I}(A) = (-1)^{S_{I_{1}}}m_{I_{2}}^{[d+1] \setminus I_{1}}
	(\underline{A}).
\]
Similarly, by the cofactor expansion from the last column,
\[
	m_{I^{c}}(B) = (-1)^{T_{J_{2}}+|J_{2}|}m_{J_{1}}^{J_{2}^{c}}
	(\underline{A}^{t}),
\]
where $T_{J_{2}}$ is the number of adjacent transpositions $(j,j+1)$ that we must apply to obtain~$J_{2}$ from $\{n-|J_{2}|+1, \ldots, n-1, n\}$, the rows of~$\underline{A}^t$ are indexed by $d+2, d+3, \ldots, n$, and the complement of~$J_2$ is taken relatively to this set of indices. Note that we have to multiply by~$(-1)^{|J_{2}|}$ because the second half of~$B$ is the negative of the identity matrix. Then 
\[
	m_{I^{c}}(B) = (-1)^{T_{J_{2}}+|J_{2}|}m_{J_{1}}^{J_{2}^{c}}
	(\underline{A}^{t})
	= (-1)^{T_{J_{2}}+|J_{2}|}m_{J_{2}^{c}}^{J_{1}}(\underline{A})
	= (-1)^{T_{J_{2}}+|J_{2}|}m_{I_{2}}^{[d+1]\setminus I_1}(\underline{A}).
\]
Thus 
\[
	m_{I}(A) = (-1)^{S_{I_{1}}+T_{J_{2}}+|J_{2}|}m_{I^{c}}(B).
\]
Now the result follows from the observations below:
\begin{enumerate}
\item If we denote by~$S'_{I_{2}}$ the number of adjacent transpositions that we have to apply to get~$I_{2}$ from $\{d+2, d+3, \ldots, d+|I_{2}|+1\}$, then $S'_{I_{2}} = T_{J_{2}}$. 
\item $|I_{1}|+|I_{2}| = d+1$. 
\item $S_{I} = S_{I_{1}}+S'_{I_{2}} + |I_{2}|(d+1-|I_{1}|) = S_{I_{1}}+S'_{I_{2}} + |I_{2}|^{2}$. 
\end{enumerate}
So 
\[
\begin{split}
	m_{I}(A) &= (-1)^{S_{I_{1}}+T_{J_{2}}+|J_{2}|}m_{I^{c}}(B)
	= (-1)^{S_{I_{1}} + S'_{I_{2}}+|J_{2}|}m_{I^{c}}(B)
	= (-1)^{S_{I}-|I_{2}|^{2}+n-d-1-|I_{2}|}m_{I^{c}}(B)\\
	&= (-1)^{S_{I}-|I_{2}|(|I_{2}|+1) +(n-d-1)}m_{I^{c}}(B)
	= (-1)^{S_{I}+(n-d-1)}m_{I^{c}}(B),
\end{split}
\]
because $|I_{2}|(|I_{2}|+1)$ is even. Observe that in this case we have~$\lambda=1$.

\par Now we prove the statement for general matrices $A$ and $B$ with nonzero minors.
There exist invertible square matrices $X,Y$ of appropriate sizes such that $A=X[\mathrm{Id}_{d+1}|\underline{A}]$ and $B=Y[\underline{A}^{t}|-\mathrm{Id}_{n-d-1}]$. Then we have that
\[
\begin{split}
m_I(A)&=m_I(X[\mathrm{Id}_{d+1}|\underline{A}])=\det(X)m_I([\mathrm{Id}_{d+1}|\underline{A}])\\
&=\det(X)(-1)^{S_I+n-d-1}m_{I^c}([\underline{A}^{t}|-\mathrm{Id}_{n-d-1}])\\
&=\frac{\det(X)}{\det(Y)}(-1)^{S_I+n-d-1}\det(Y)m_{I^c}([\underline{A}^{t}|-\mathrm{Id}_{n-d-1}])\\
&=\frac{\det(X)}{\det(Y)}(-1)^{S_I+n-d-1}m_{I^c}(B),
\end{split}
\]
which is the required statement where $\lambda=\det(X)/\det(Y)$.
\end{proof}

We next turn to defining a projective Gale transform.  Fix homogeneous coordinates on~$(\PP^d)^n$, and for any matrix~$A\in M_{(d+1)\times n}$ that does not have the zero vector as a column, let~$\PP A \in (\PP^d)^n$ denote the projective point configuration whose homogeneous coordinates are the columns of~$A$.  There is a crucial geometric property necessary for the projective Gale transform to be well-defined.

\begin{definition}
A point configuration $\mathbf{p} = (p_{1}, \ldots, p_{n}) \in (\PP^{d})^{n}$ is \emph{strongly non-degenerate} if for any hyperplane~$H \subseteq \PP^{d}$ there are at least two points of the configuration that do not lie on~$H$.
\end{definition}

\begin{remark}
If a point configuration $\mathbf{p} = (p_{1}, \ldots, p_{n}) \in(\PP^d)^n$ is automorphism-free (meaning that there is no non-trivial automorphism of~$\PP^d$ fixing all the~$p_i$), then $\mathbf{p}$ is strongly non degenerate (see also \cite[Proposition 3.2]{GG18}). Notice that the converse of this statement is false. As a counterexample, consider the two skew lines $Z(X_0,X_1),Z(X_2,X_3)$ in $\mathbb{P}^3$ and a configuration of $n\geq6$ distinct points on these two lines with at least three points on each line. This $n$-point configuration is strongly non-degenerate, but it is not automorphism-free. Non-trivial automorphisms of $\mathbb{P}^3$ fixing the $n$-points are
\begin{displaymath}
\left( \begin{array}{cccc}
\lambda&0&0&0\\
0&\lambda&0&0\\
0&0&\mu&0\\
0&0&0&\mu
\end{array} \right),
\end{displaymath}
for all choices of $\lambda,\mu\in\Bbbk\setminus\{0\}$, $\lambda\neq\mu$.
\end{remark}

\begin{definition}\label{def:Galetransform}
Let~$\mathbf{p} = (p_{1}, \ldots, p_{n}) \in (\PP^{d})^{n}$ be a strongly non-degenerate point configuration.  The set of \emph{Gale transforms} of~$\mathbf{p}$ is
\[\widetilde{G}(\mathbf{p}) := \{\PP B \in (\PP^{n-d-2})^n ~|~\PP A = \mathbf{p}\text{ and }B\in \widetilde{G}(A)\}.\] 
\end{definition}

For this definition to be well-defined, we need the following:

\begin{lemma}\label{lem:nonzerocol}
Let~$\mathbf{p} \in (\PP^d)^n$ be a strongly non-degenerate configuration.  Then (1) any matrix~$A\in M_{(d+1)\times n}$ satisfying $\PP A = \mathbf{p}$ is full rank, and (2) given such an~$A$, each column of any matrix~$B\in \widetilde{G}(A)$ is nonzero.
\end{lemma}

\begin{proof}
Statement~(1) is obvious, since strongly non-degenerate implies non-degenerate, so we turn to~(2).  Without loss of generality we may assume that $A = [\mathrm{Id}_{d+1} | \underline{A}]$, and so as noted earlier every matrix in~$\widetilde{G}(A)$ is in the~$\GL_{n-d-1}$-orbit of the matrix $[\underline{A}^{t}|-\mathrm{Id}_{n-d-1}]$.  Thus it suffices to show that each column of this latter matrix is nonzero, or equivalently that each row of~$\underline{A}$ is nonzero.  If, say, the~$j$-th row of~$\underline{A}$ is zero, then the projectivization of the~$j$-th column of $A$ yields the only point of~$\mathbf{p}$ outside the hyperplane $H\subseteq \PP^d$ defined by the vanishing of the~$j$-th coordinate, contradicting the strongly non-degenerate hypothesis. 
\end{proof}

Now that we have established that the set of Gale transforms $\widetilde{G}(\mathbf{p}) \subseteq (\PP^d)^n$ is well-defined on strongly non-degenerate configurations, we can extend Remark~\ref{rem:GaleOrbits} to this projective setting.

\begin{proposition}\label{prop:3GaleProps}
Let~$\mathbf{p} \in (\PP^d)^n$ be a strongly non-degenerate configuration.  Then:
\begin{enumerate}
\item $\PGL_{n-d-1}$ acts transitively on~$\widetilde{G}(\mathbf{p})$,
\item $\widetilde{G}(\cdot)$ is constant on~$\PGL_{d+1}$-orbits, and
\item Each configuration in $\widetilde{G}(\mathbf{p})$ is strongly non-degenerate.
\end{enumerate}
\end{proposition}

\begin{proof}
Without loss of generality let us suppose that the first $d+1$~points of $\mathbf{p}$ are linearly independent.  Then the $\GL_{d+1}$-orbit of each matrix~$A$ satisfying $\PP A = \mathbf{p}$ contains a matrix of the form $[\mathrm{Id}_{d+1}|\underline{A}]$, and the torus~$(\Bbbk^*)^{n-d-1}$ acts transitively on the set of such matrices by rescaling the columns of~$\underline{A}$.  The set of affine Gale transforms $\widetilde{G}([\mathrm{Id}_{d+1}|\underline{A}])$ coincides with the~$\GL_{n-d-1}$-orbit of the matrix $[\underline{A}^t|-\mathrm{Id}_{n-d-1}]$.  Since the action of~$(\Bbbk^*)^{n-d-1}$ rescaling the rows of~$\underline{A}^t$ extends to the left-multiplication action of~$\GL_{n-d-1}$, we see that this set of affine Gale transforms is unaffected by rescaling the columns of~$\underline{A}$.  This implies assertion~(1).  Assertion~(2) follows immediately from the corresponding fact about the affine Gale transform.  To prove~(3) it suffices, by~(1), to show that a single configuration in~$\widetilde{G}(\mathbf{p})$ is strongly non-degenerate.  So suppose to the contrary that the $(n-d-1)\times n$ matrix $[\underline{A}^t|-\mathrm{Id}_{n-d-1}]$ has a submatrix of size $(n-d-1)\times (n-1)$ that is not full rank.  Clearly the missing column is then among the last block of~$n-d-1$ columns.  If it is the~$j$-th column in that negated identity matrix, then by considering the vanishing minors obtained by using the remaining $n-d-2$~columns of this identity block and each column of~$\underline{A}^t$ separately, we deduce that the~$j$-th row of $\underline{A}^t$ is zero.  But this means the~$j$-th column of~$\underline{A}$ is zero, which is impossible because $\PP A=\mathbf{p}$.
\end{proof}

The preceding result says that the Gale transform sends the projective equivalence class of a strongly non-degenerate configuration of~$n$~points in~$\PP^d$ to the projective equivalence class of a strongly non-degenerate configuration of~$n$~points in~$\PP^{n-d-2}$.  This is a precise formulation of the usual geometric statement of the Gale transform.  Moreover, it is clear from the definition that this is involutive, meaning that applying the Gale transform to the projective equivalence class $\widetilde{G}(\mathbf{p})$ results in the projective equivalence class of~$\mathbf{p}$.

A classical fact about the Gale transform, first observed by Goppa in the context of coding theory, is that any configuration of distinct points supported on a rational normal curve in~$\PP^d$ (which is necessarily a strongly non-degenerate configuration, since such points are in general linear position and $n \ge d+2$) is sent to a configuration of distinct points supported on a rational normal curve in~$\PP^{n-d-2}$.  In our language this is the following:

\begin{proposition}[\protect{\cite[Corollary 3.2]{EP00}}]\label{prop:GTandRNC}
If $\mathbf{p} \in U_{d, n} \subseteq (\PP^d)^n$, then $\widetilde{G}(\mathbf{p}) \subseteq U_{n-d-2, n}$. 
\end{proposition}

From this fact, together with continuity and involutivity of the Gale transform, we immediately obtain the following result (cf.,~\cite[\S6.2]{Gia13}):

\begin{corollary}\label{cor:GTVdn}
If $\mathbf{p} \in V_{d, n} \subseteq (\PP^d)^n$ is strongly non-degenerate, then $\widetilde{G}(\mathbf{p}) \subseteq V_{n-d-2, n}$. 
\end{corollary}


\subsection{The case $d\geq3$}

Since we already analyzed the Veronese compactification~$V_{d,n} \subseteq (\PP^d)^n$ for~$d=2$, we now assume~$d \ge 3$.  Recall that $V_{d,n} = (\PP^d)^n$ for all~$n \le d+3$, so the first non-trivial case is $V_{d,d+4}$.  

\begin{definition}
For each~$I\in \binom{[d+4]}{6}$, let~$\psi_I$ be the polynomial in the maximal minors of a matrix of coordinates for~$(\PP^d)^{d+4}$ obtained by taking the polynomial~$\phi_I$ from Definition~\ref{def:W2n}, which is a polynomial in the maximal minors of a matrix of coordinates for~$(\PP^2)^{d+4}$ (see Remark~\ref{rem:bracket}), and applying the transformation \[m_J \mapsto (-1)^{S_J + d + 1}m_{J^c}.\]  (The notation~$S_J$ was introduced in Definition \ref{def:S}.)  Let $W_{d, d+4} \subseteq (\PP^{d})^{d+4}$ be the closed subscheme defined by~$\{\psi_{I}\}$ for all~$I\in \binom{[d+4]}{6}$. 
\end{definition}

\begin{remark}
The number of polynomials $\psi_{I}$ is $\binom{d+4}{6}$, and their multi-degrees are permutations of~$(2,2,2,2,2,2,4,\ldots,4)$.
\end{remark}


\begin{example}\label{ex:equationsV37}
For~$d=3$, by Remark~\ref{rem:bracket} we have
\[
	\phi_{[6]} = |123||145||246||356| 
	- |124||135||236||456|,
\]
and so
\[
	\psi_{[6]} = |4567||2367||1357||1247| - |3567||2467||1457||1237|.
\]
One can similarly work out the remaining six polynomials defining $W_{3,7}$.  We show below (Corollary~\ref{cor:equationsV37V38V48}) that these seven polynomials cut out~$V_{3,7}$ set-theoretically, i.e.~$V_{3,7}=W_{3,7}$ as sets.  The same seven equations were obtained over a century ago by White from a different construction and for a different purpose~\cite{Whi15}.  \end{example}

Recall that for the~$d=2$~conic case studied in~\S\ref{ssec:equationsV2n}, we first defined $W_{2,6}$ as the vanishing of a determinant and then pulled this polynomial back along all the forgetful maps to define a subscheme~$W_{2,n}$ that we proved coincides with the Veronese compactification~$V_{2,n}$.  In this subsection we mimic that story for~$d\ge 3$ by pulling back the equations for~$W_{d,d+4}$ produced above to define a subscheme~$W_{d,n}$ for all~$n \ge d+4$.  However, we shall see below that there are significant twists to this story when~$d \ge 3$.  

\emph{Note}: For the remainder of this section we assume $d\ge 3$ and~$n \ge d+4$.

\begin{definition}
For each~$J \in\binom{[n]}{d+4}$, there is a projection morphism $\pi_{J} \colon (\PP^{d})^{n} \twoheadrightarrow (\PP^{d})^{d+4}$ sending $(p_{i})$ to~$(p_{i})_{i \in J}$. For each defining polynomial $\psi_{I}$ for~$W_{d, d+4}$, let~$\psi_{I, J} := \pi_{J}^{*}\psi_{I}$. Let~$W_{d, n} \subseteq (\PP^{d})^{n}$ be the closed subscheme defined by~$\{\psi_{I, J}\}$ for all possible~$I,J$. 
\end{definition}

The first basic fact about this scheme is that it contains the Veronese compactification:

\begin{lemma}\label{lem:VdninWdn}
We have $V_{d, n} \subseteq W_{d, n}$ as schemes.
\end{lemma}

\begin{proof}
Since~$V_{d,n}$ is reduced, it suffices to show this containment set-theoretically.  Let~$\mathbf{p} = (p_{i}) \in U_{d, n}$. For any $J \in\binom{[n]}{d+4}$ we have $\mathbf{p}_{J} := (p_{i})_{i \in J} \in U_{d, d+4}$.  By Proposition~\ref{prop:GTandRNC} we then have $\widetilde{G}(\mathbf{p}_{J}) \subseteq U_{2, d+4}$, so $\phi_I$ vanishes on~$\widetilde{G}(\mathbf{p}_{J})$ for each~$I \in \binom{[d+4]}{6}$.  It then follows from Proposition~\ref{prop:Galetransformminor} and the definition of~$\psi_I$ that $\psi_I$ vanishes at~$\mathbf{p}_J$.  Since this holds for all~$I$ and $J$, we have $U_{d, n} \subseteq W_{d, n}$ and hence, by continuity, $V_{d, n} \subseteq W_{d, n}$.
\end{proof}

Recall that $Y_{d, n} \subseteq (\PP^{d})^{n}$ is the space of degenerate point configurations (Definition~\ref{def:Ydn}). 

\begin{lemma}\label{lem:YdninWdn}
We have $Y_{d, n} \subseteq W_{d, n}$ as schemes.
\end{lemma}

\begin{proof}
By definition the ideal for~$Y_{d,n}$ is generated by the maximal minors of a matrix of coordinates for~$(\PP^d)^n$, and by construction the polynomials~$\psi_{I, J}$ generating the ideal for~$W_{d,n}$ are all polynomials in these maximal minors.
\end{proof}

Thus we have a scheme-theoretic containment of the scheme-theoretic union: 
\[V_{d,n} \cup Y_{d,n} \subseteq W_{d,n}.\]  By putting together what we already know about the Gale transform, we can show that the opposite inclusion holds set-theoretically in the base case~$n=d+4$:

\begin{theorem}\label{thm:Wdd4inYdd4}
Set-theoretically, we have $W_{d, d+4} = V_{d, d+4} \cup Y_{d, d+4}$.
\end{theorem}

\begin{proof}
Fix~$\mathbf{p} \in W_{d, d+4}$.  If~$\mathbf{p}$ is degenerate, then $\mathbf{p} \in Y_{d, d+4}$ and we are done. If~$\mathbf{p}$ is strongly non-degenerate, then Proposition~\ref{prop:Galetransformminor} and the definition of~$W_{d, d+4}$ implies that each Gale transform $\widetilde{G}(\mathbf{p})$ lies in~$V_{2, d+4}$, and each Gale transform is strongly non-degenerate by Proposition~\ref{prop:3GaleProps}, so $\widetilde{G}(\widetilde{G}(\mathbf{p})) \subseteq V_{d, d+4}$ by Corollary~\ref{cor:GTVdn}; since~$\mathbf{p} \in \widetilde{G}(\widetilde{G}(\mathbf{p}))$, we are done.  So we can assume that $\mathbf{p}$ is non-degenerate but not strongly non-degenerate. 

Let~$A$ be a matrix associated to~$\mathbf{p}$, i.e., $\PP A = \mathbf{p}$.  Since all loci in question are $\PGL_{d+1}$-invariant, the fact that the points in~$\mathbf{p}$ span $\PP^d$ yet there is a hyperplane $H \subseteq \PP^d$ containing all but one of the points means we can assume without loss of generality that $A = [\mathrm{Id}_{d+1} | \underline{A}]$ where the first $e \ge 1$ rows of~$\underline{A}$, and no others, are zero.  Then an affine Gale transform of~$A$ is given by~$B = [\underline{A}^{t}|-\mathrm{Id}_{3}]$, a $3\times (d+4)$ matrix whose zero columns are precisely the first $e$ columns. Let $B'$ be the $3 \times (d+4-e)$ matrix obtained by removing these zero columns, and let $\mathbf{q} = \PP B' \in (\PP^2)^{d+4-e}$ be the associated point configuration.  Since $\mathbf{p}\in W_{d,d+4}$, we know all of the $\psi_{I,J}$ vanish at $A$, so by Proposition~\ref{prop:Galetransformminor} the columns of~$B$, and hence of~$B'$, are vectors lying on the affine cone over a plane conic.  Thus after projectivization we obtain  $\mathbf{q}\in V_{2,d+4-e}$.  We claim that $\mathbf{q}$ is strongly non-degenerate.  Indeed, if we remove the first $e$ rows and columns of~$A$ then the associated configuration $\mathbf{r} \in (\PP^{d-e})^{d+4-e}$ is clearly strongly non-degenerate and has $\mathbf{q}$ as a Gale transform, $\mathbf{q}\in \widetilde{G}(\mathbf{r})$, so the claim follows from Proposition~\ref{prop:3GaleProps}(3).  

Since $\mathbf{q}\in V_{2,d+4-e}$ is strongly non-degenerate, Corollary~\ref{cor:GTVdn} tells us that $\widetilde{G}(\mathbf{q}) \subseteq V_{d-e, d-e+4}$, and involutivity of the Gale transform implies $\mathbf{r}\in \widetilde{G}(\mathbf{q})$, so we see that $\mathbf{r} \in V_{d-e,d-e+4}$.  Thus, there is a quasi-Veronese curve~$C$ of degree~$d-e$ in the~$\PP^{d-e} \subseteq \PP^d$ defined by the vanishing of the first $e$ coordinates that passes through the points $p_{e+1}, p_{e+2}, \ldots, p_{d+4}$. Let $C'$ be the union of~$C$ and the $e$ lines obtained by connecting each $p_i$, for $1 \le i \le e$, to any point of~$C$.  Then $C'$ is a degree~$d$ quasi-Veronese curve in~$\PP^d$ because the points $p_1,\ldots,p_e$ are the standard coordinate points outside of the linear subspace $\mathbb{P}^{d-e}$. Since $C'$ passes through all $d+4$ points of~$\mathbf{p}$, we have that $\mathbf{p} \in V_{d, d+4}$ as desired.
\end{proof}

Before turning to the question of whether the preceding proposition extends beyond~$n=d+4$, it is convenient to introduce the following geometric construction:

\begin{definition}
For $\mathbf{p}\in (\PP^d)^n$, let the \emph{Veronese envelope} $E_{\mathbf{p}} \subseteq \PP^{d}$ be the union of all quasi-Veronese curves passing through all points of~$\mathbf{p}$. 
\end{definition}

Note that the order of the points does not affect the Veronese envelope, and $E_{\mathbf{p}} \ne \emptyset$ if and only if $\mathbf{p}\in V_{d,n}$.  Moreover, for any subset~$I \subseteq [n]$ and corresponding projection map $\pi_I \colon (\PP^d)^n \rightarrow (\PP^d)^{|I|}$, we have $E_{\mathbf{p}} \subseteq E_{\pi_I(\mathbf{p})}$.

We will need the following result shortly for an inductive argument:

\begin{lemma}\label{lem:envelopeinduction}
Suppose $\mathbf{p} \in W_{d, n} \setminus Y_{d,n}$ and $\pi_{[n-1]}(\mathbf{p}) \in V_{d, n-1}$.  If the containment
\begin{equation}\label{eqn:equalityofenvelopes}
	E_{\pi_{[n-1]}(\mathbf{p})} \subseteq \bigcap_{I}
	E_{\pi_I(\mathbf{p})}
\end{equation}
is an equality, where the intersection is over $I \in \binom{[n-1]}{d+3}$ such that $\pi_I(\mathbf{p}) \notin Y_{d,d+3}$, then $\mathbf{p} \in V_{d, n}$.
\end{lemma}

\begin{proof}
Write $\mathbf{q} = \pi_I(\mathbf{p})$ for a subset $I\in\binom{[n-1]}{d+3}$ such that this projection is non-degenerate, and let~$\mathbf{q}' = \pi_{I\sqcup \{n\}}(\mathbf{p})$. Since $\mathbf{p} \in W_{d, n}$, by definition we have $\mathbf{q}' \in W_{d, d+4}$, and so~$\mathbf{q}' \in V_{d, d+4} \cup Y_{d, d+4}$ by Theorem~\ref{thm:Wdd4inYdd4}.  Since $\mathbf{q}$ is non-degenerate, so must be $\mathbf{q}'$, and so~$\mathbf{q}'\in V_{d,d+4}$.  In particular, there is a quasi-Veronese curve through $p_n$ and all points of~$\mathbf{q}$, so $p_{n} \in E_{\mathbf{q}}$.  By letting the index set~$I$ vary, we deduce that $p_n$ is in the right side, and hence also the left side, of the hypothesized equality~\eqref{eqn:equalityofenvelopes}.  Thus there is a quasi-Veronese curve through~$p_n$ and all~$p_i$ for~$1 \le i \le n-1$, or in other words, $\mathbf{p}\in V_{d,n}$.
\end{proof}

\begin{theorem}\label{thm:W3ninV3n}
Set-theoretically, we have $W_{3, n} = V_{3, n} \cup Y_{3, n}$.
\end{theorem}

\begin{proof}
We use induction on~$n$. The base case, $n = 7$, is Theorem~\ref{thm:Wdd4inYdd4}, and for arbitrary $n$ one containment is given by Lemmas~\ref{lem:VdninWdn}~and~\ref{lem:YdninWdn}.  Let $\mathbf{p} \in W_{3, n} \setminus Y_{3,n}$; then we must show that $\mathbf{p}\in V_{3,n}$.  By the inductive hypothesis, we may assume that all points of~$\mathbf{p}$ are distinct.  We claim that by reordering the points if necessary, we may also assume that $\pi_{[n-1]}(\mathbf{p}) \notin Y_{3,n-1}$.  Indeed, if there is a plane $H_1 \subseteq \PP^3$ containing all but one point of~$\mathbf{p}$ and another plane $H_2 \subseteq \PP^3$ containing all but a different point, then the line $H_1\cap H_2 \subseteq \PP^3$ contains all but two points, so the union of this line with the line between those two remaining points and any third line joining these two lines yields a quasi-Veronese curve through all~$n$~points, thus showing that $\mathbf{p}\in V_{3,n}$ thereby verifying the claim.  So, we have $\pi_{[n-1]}(\mathbf{p}) \in W_{3,n-1} \setminus Y_{3,n-1}$ and hence by the inductive hypothesis, $\pi_{[n-1]}(\mathbf{p})\in V_{3,n-1}$.  

If~$\pi_{[n-1]}(\mathbf{p})$ contains one of the non-degenerate six-point sub-configurations $\pi_I(\mathbf{p})$, $I\in\binom{[n-1]}{6}$, depicted in Figure~\ref{fig:sixpoint}, then there is a unique quasi-Veronese curve through $\pi_{I}(\mathbf{p})$ and so~$E_{\pi_{[n-1]}(\mathbf{p})} = E_{\pi_{I}(\mathbf{p})}$.  In this case we may apply Lemma \ref{lem:envelopeinduction} to conclude that $\mathbf{p}\in V_{3,n}$.

\begin{figure}[!ht]
\centering
\includegraphics[scale=0.55,valign=t]{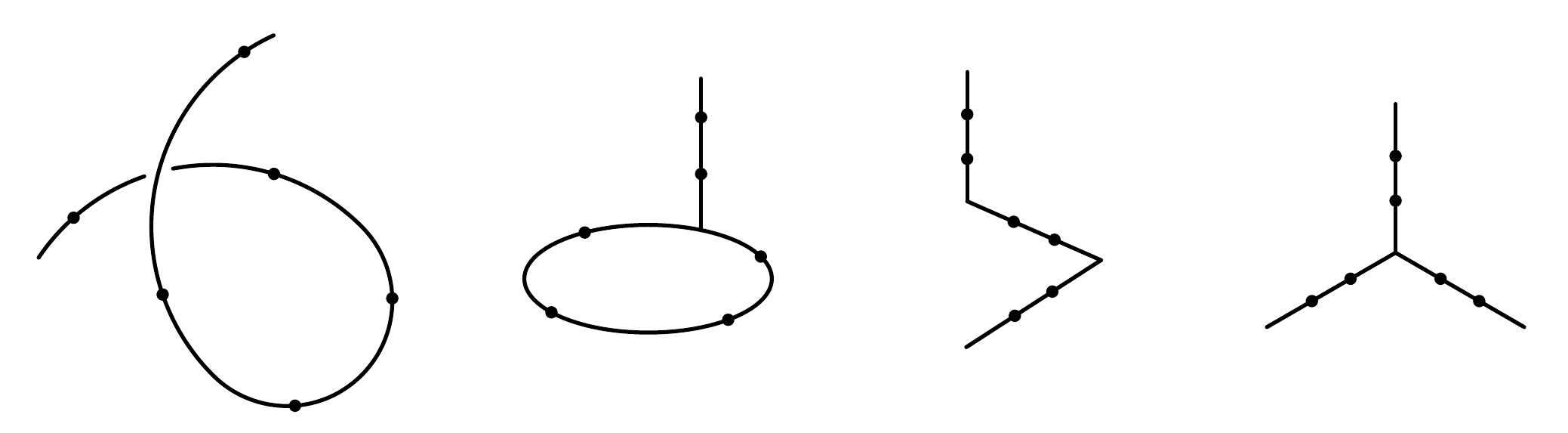}
\caption{Six-point configurations whose Veronese envelope is a curve.}
\label{fig:sixpoint}
\end{figure}

Thus it is sufficient to consider the point configurations without such sub-configurations. Table~\ref{tbl:configurations} and Figure~\ref{fig:exceptionalcase} show all remaining possibilities of~$\pi_{[n-1]}(\mathbf{p})$. Table~\ref{tbl:configurations} also shows the Veronese envelopes $E_{\pi_{[n-1]}(\mathbf{p})}$, and for each one it illustrates a single non-degenerate six-point sub-configuration $\pi_{I}(\mathbf{p})$ whose Veronese envelope equals that of~$\pi_{[n-1]}(\mathbf{p})$. The cases in Figure~\ref{fig:exceptionalcase} are isolated from the others because there $E_{\pi_{[n-1]}(\mathbf{p})}$, which is the curve itself, is not the Veronese envelope of a single non-degenerate six-point sub-configuration but instead is the intersection of two such Veronese envelopes.  In all such cases we can apply again Lemma~\ref{lem:envelopeinduction} to deduce that $\mathbf{p} \in V_{3, n}$. Let us explain how we computed this list. First of all, one can consider all possible quasi-Veronese curves and list all possible ways $\pi_{[n-1]}(\mathbf{p})$ can distribute on them. Then many of these cases are redundant because $\pi_{[n-1]}(\mathbf{p})$ can be supported on different quasi-Veronese curves. The cases shown are what is left after eliminating the redundant cases. Observe that there are choices involved in this process: for instance, in the second case in Table~\ref{tbl:configurations}, we have that $\pi_{[n-1]}(\mathbf{p})$ could also be supported on a quasi-Veronese curve consisting of three concurrent lines.
\end{proof}

\begin{table}[!ht]
\centering
\begin{tabular}{|>{\centering\arraybackslash}m{1.5cm}||>{\centering\arraybackslash}m{2.1cm}|>{\centering\arraybackslash}m{2.1cm}|>{\centering\arraybackslash}m{2.1cm}|>{\centering\arraybackslash}m{2.1cm}|>{\centering\arraybackslash}m{2.1cm}|}\hline 
$\pi_{[n-1]}(\mathbf{p})$ 
& \includegraphics[scale=0.6]{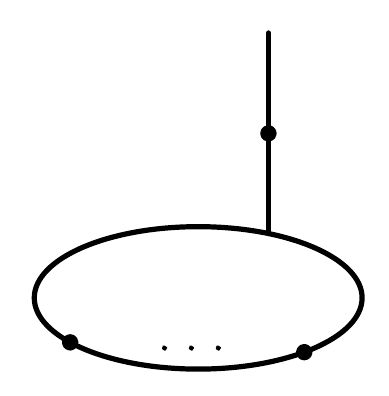} 
& \includegraphics[scale=0.6]{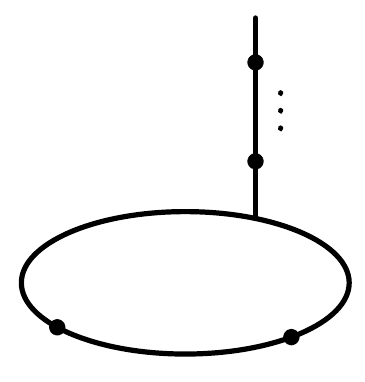}
& \includegraphics[scale=0.6]{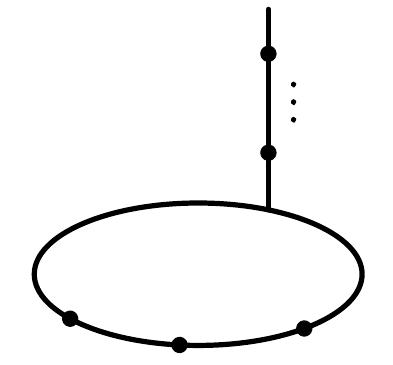}
& \includegraphics[scale=0.6]{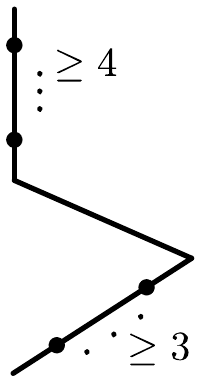}
\\ \hline
$E_{\pi_{[n-1]}(\mathbf{p})}$ & cone over a conic & $\PP^{3}$ & $\PP^{2} \cup \PP^{1}$ & $\PP^{3}$  \\ \hline
$\pi_{I}(\mathbf{p})$ 
& \includegraphics[scale=0.6]{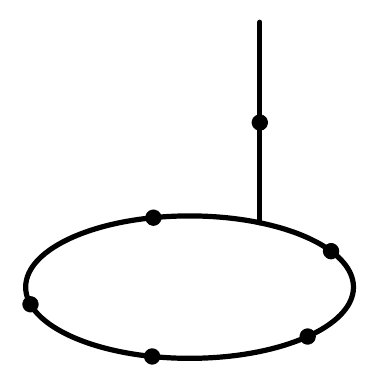} 
& \includegraphics[scale=0.6]{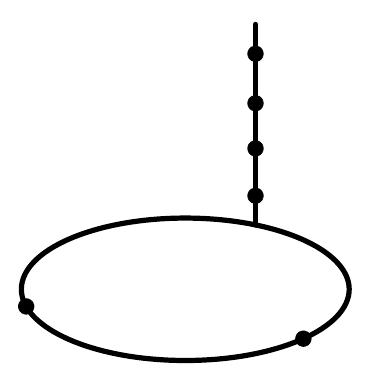} 
& \includegraphics[scale=0.6]{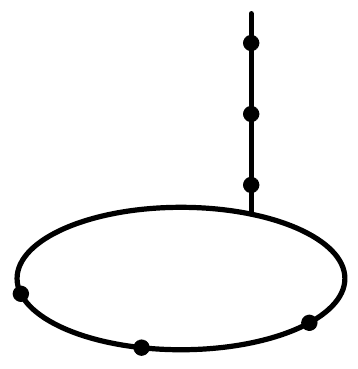} 
& \includegraphics[scale=0.6]{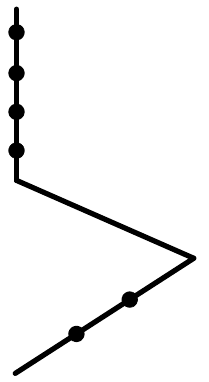}
\\ \hline \hline
$\pi_{[n-1]}(\mathbf{p})$
& \includegraphics[scale=0.6]{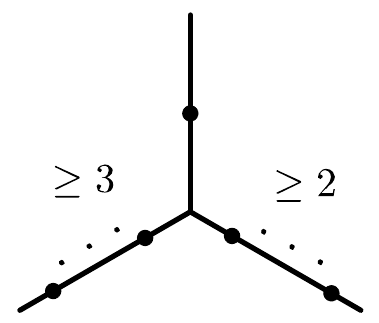} 
& \includegraphics[scale=0.6]{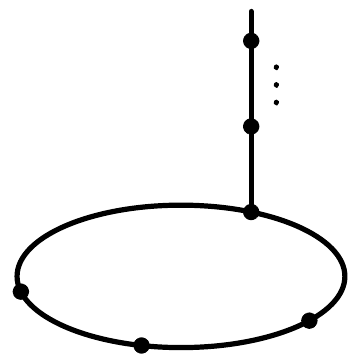}
& \includegraphics[scale=0.6]{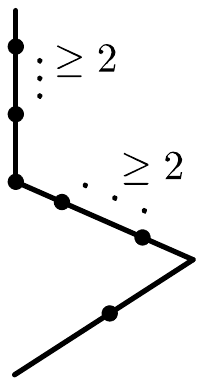}
& \\ \hline
$E_{\pi_{[n-1]}(\mathbf{p})}$ & $\PP^{2} \cup \PP^{2}$ & $\PP^{2}\cup\PP^{1}$ & $\PP^{2}\cup\PP^{2}$ &  \\ \hline
$\pi_{I}(\mathbf{p})$
& \includegraphics[scale=0.6]{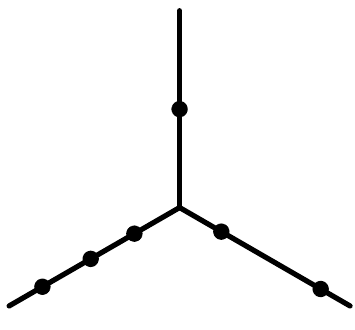} 
& \includegraphics[scale=0.6]{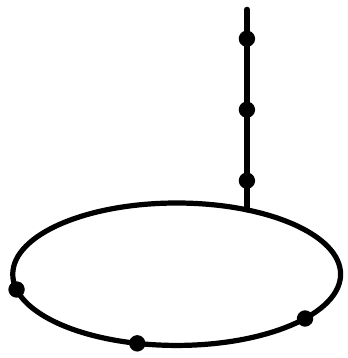}
& \includegraphics[scale=0.6]{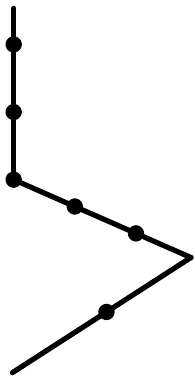}
& \\ \hline
\end{tabular}
\caption{Configurations of $n-1\geq7$ points not containing the six-point sub-configurations in Figure~\ref{fig:sixpoint}, and for each one a six-point sub-configuration with the same Veronese envelope.}
\label{tbl:configurations}
\end{table}

\begin{figure}[!ht]
\centering
\begin{tabular}{ccccc}
$E\Big($\includegraphics[scale=0.6, valign=m]{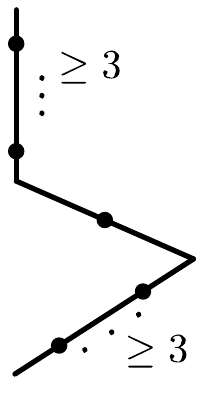}$\Big)$ & = &
$E\Big($\includegraphics[scale=0.6, valign=m]{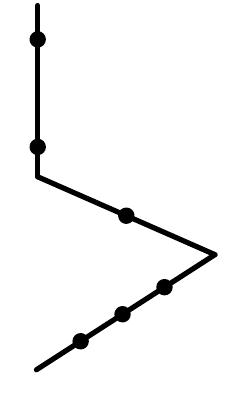}$\Big)$ & $\cap$ &
$E\Big($\includegraphics[scale=0.6, valign=m]{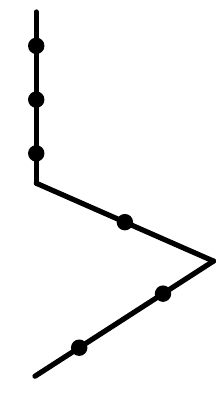}$\Big)$\\
$E\Big($\includegraphics[scale=0.6, valign=m]{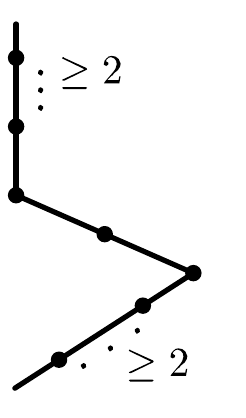}$\Big)$ & = &
$E\Big($\includegraphics[scale=0.6, valign=m]{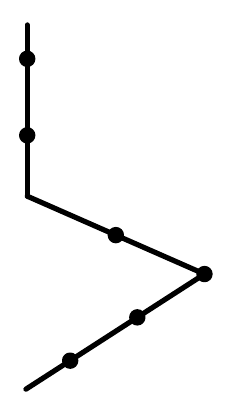}$\Big)$ & $\cap$ &
$E\Big($\includegraphics[scale=0.6, valign=m]{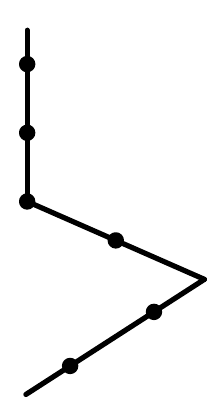}$\Big)$\\
$E\Big($\includegraphics[scale=0.6, valign=m]{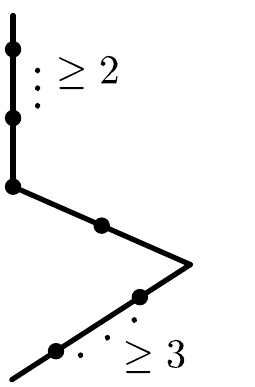}$\Big)$ & = &
$E\Big($\includegraphics[scale=0.6, valign=m]{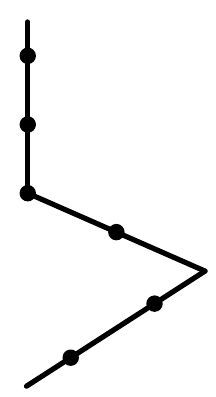}$\Big)$ & $\cap$ &
$E\Big($\includegraphics[scale=0.6, valign=m]{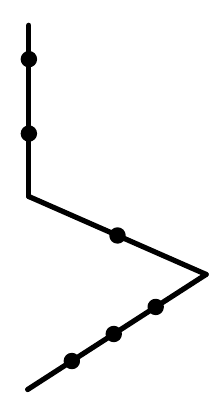}$\Big)$
\end{tabular}
\caption{On the left, point configurations $\pi_{[n-1]}(\mathbf{p})$ whose Veronese envelopes, the curve itself, are the intersection of the Veronese envelopes of the two six-point sub-configurations on the right. Here $E$ of a pointed curve refers to the Veronese envelope of the point configuration.}
\label{fig:exceptionalcase}
\end{figure}

We expect that the preceding result extends to all higher dimensions:

\begin{conjecture}\label{conj:twocomps}
For all~$d \ge 4$ and $n \ge d+5$, set-theoretically we have $W_{d, n} = V_{d, n} \cup Y_{d, n}$.
\end{conjecture}

The combinatorics of our proof technique in the~$d=3$~case, which involves a brute-force enumeration of many cases, quickly becomes unwieldy; a more conceptual proof would be desirable and is likely necessary to tackle this conjecture.

\begin{remark}
Let us pause a moment to recall where things stand now.  We found equations that cut out $V_{2,n}$ scheme-theoretically and classified which of these equations can be omitted to still cut out $V_{2,n}$ set-theoretically.  For~$d \ge 3$ we used the Gale transform to find equations of a scheme that set-theoretically we conjecture satisfies $W_{d,n} = V_{d,n} \cup Y_{d,n}$, where $Y_{d,n}$ is the locus of degenerate point configurations.  We proved this conjecture when~$d \ge 3$ and~$n=d+4$, and also when~$d=3$ and~$n \ge d+4$.  The variety~$V_{d,n}$ is irreducible (Lemma~\ref{lem:Vdnirreducible}), and it is well-known that the determinantal variety~$Y_{d,n}$ is irreducible.  Thus, for each pair~$(d,n)$ the equality $W_{d,n} = V_{d,n} \cup Y_{d,n}$ means that either (1)~$Y_{d,n} \subseteq V_{d,n}$, in which case the Gale equations defining $W_{d,n}$ do indeed cut out the Veronese compactification~$V_{d,n}$ set-theoretically, or (2)~the variety $(W_{d,n})_{red}$ has precisely two irreducible components, namely the Veronese compactification~$V_{d,n}$ and the degenerate point configuration loci~$Y_{d,n}$.

Consider a case where $(W_{d,n})_{red}$ does have these two irreducible components.  In order to set-theoretically cut $W_{d,n}$ down to the irreducible component we care about, $V_{d,n}$, we need polynomials that vanish on~$V_{d,n}$ but are nonzero at a general point of~$Y_{d,n}$.  However, by the first fundamental theorem of invariant theory, any~$\SL_{d+1}$-invariant multi-homogeneous polynomial on~$(\PP^d)^n$ is a polynomial in the maximal minors of the matrix of homogeneous coordinates, so such a polynomial must vanish on~$Y_{d,n}$ since the latter is defined by the vanishing of these minors.  Therefore, the polynomials that cut $W_{d,n}$ down to~$V_{d,n}$ \emph{cannot} be invariant with respect to projectivities.  This largely rules out the possibility of constructing them in a geometric manner and is part of the reason we believe it will be quite difficult to find a complete set of equations for~$V_{d,n}$, even set-theoretically in the case~$d=3$.
\end{remark}

\begin{proposition}\label{prop:irreducible}
Suppose that~$d \ge 3$ and~$n \ge d+4$. We have $Y_{d, n} \subseteq V_{d, n}$ if and only if \[(d, n) \in \{(3, 7), (3, 8), (4, 8)\}.\] 
\end{proposition}

\begin{proof}
Assume $(d,n)\neq(3,7),(3,8),(4,8)$. We claim $Y_{d,n}\not\subseteq V_{d,n}$ for dimension reasons:
\[\dim(Y_{d,n})=nd-n+d\geq d^2+2d+n-3=\dim(V_{d,n}).\]
The equality on the left follows from a simple parameter count and the equality on the right is Lemma~\ref{lem:Vdnirreducible} (or a simple parameter count as well), so it suffices to establish the inequality in the middle. If~$d=3,4$ and~$n\geq9$ then this inequality is clear, so suppose $d\geq 5$.  By rewriting the inequality as~$n(d-2)\geq d^2+d-3$, we see that it holds since \[n(d-2)\geq (d+4)(d-2)\geq d^2+d-3,\] where the first inequality here uses $n \ge d+4$ and the second one uses $d\geq5$.

For the converse, first suppose that $d = 3$ and~$n = 8$. For any 8 points~$\mathbf{p}$ in~$\PP^{2}$, there is a pencil of cubic curves passing through them. Because the discriminant of singular cubic curves is a hypersurface in~$|\cO_{\PP^2}(3)|$, there is a singular cubic curve passing through~$\mathbf{p}$. By embedding~$\PP^2$ as a plane in~$\PP^3$ and adding an embedded spatial point to this singular cubic at one of its singular points, it follows from~\cite{PS85} that there is a family of twisted cubics whose flat limit is this singular plane cubic curve (with embedded point), and so $\mathbf{p}$ is in the closure of~$U_{3, 8}$.  Thus $Y_{3,8} \subseteq V_{3, 8}$, and hence $Y_{3,7} \subseteq V_{3,7}$ as well.

Lastly, let~$\mathbf{p} \in Y_{4, 8}$ be a general degenerate point configuration, say~$\mathbf{p} \subseteq \PP^{3} \subseteq \PP^{4}$. By applying Gromov-Witten theory, in particular the Reconstruction Theorem~(\cite[\S4.4]{KV07} for characteristic zero, \cite{Pom12} for positive characteristic), we can count the number of rational curves of degree~4 passing through~$\mathbf{p}$, which is four (see~\cite[Equation~(3.9)]{DFI95}). In characteristic zero, we may also rely on interpolation theory to guarantee the existence (but not the precise number) of such a rational curve (see~\cite{ALY15}). Thus $V_{4,8}$ contains a general point of~$Y_{4,8}$ and hence it contains all of~$Y_{4,8}$.
\end{proof}

\begin{corollary}\label{cor:equationsV37V38V48}
We have explicit set-theoretic equations for~$V_{3,7},V_{3,8}$, and~$V_{4,8}$.
\end{corollary}

\begin{proof}
In these cases we have $W_{d,n} = V_{d,n}\cup Y_{d,n}$ by Theorems~\ref{thm:Wdd4inYdd4}~and~\ref{thm:W3ninV3n}, and we have $Y_{d,n} \subseteq V_{d,n}$ by Proposition~\ref{prop:irreducible}, so~$W_{d,n} = V_{d,n}$.
\end{proof}

This also brings up the second reason why finding a complete set of defining equations for~$V_{d, n}$ seems to be a very difficult problem in general:

\begin{corollary}
The pull-backs of all polynomials vanishing on $V_{3, 7} = W_{3, 7}$ do not set-theoretically cut out~$V_{3, n}$ for any~$n \ge 9$.
\end{corollary}

\begin{proof}
By Proposition~\ref{prop:irreducible}, a general configuration of~$n \ge 9$ points in~$\PP^2 \subseteq \PP^3$ is outside of~$V_{3,n}$ even though its image under each projection map $(\PP^3)^n \twoheadrightarrow (\PP^3)^7$ is in~$Y_{3,7} \subseteq V_{3,7}$.
\end{proof}

A third indicator of the sudden increase in complexity when passing from~$d=2$ to~$d \ge 3$ is that $V_{d,n}$ in the latter case does not seem to have nice geometric/arithmetic properties as it does in the former.  For instance:

\begin{proposition}
\label{notnormalindegree3}
For~$n \ge 8$, $V_{3, n}$ is not normal.
\end{proposition}

\begin{proof}
Consider a triangular configuration of three lines $\ell_{1}, \ell_{2}, \ell_{3}$ on a~$\PP^{2} \subseteq \PP^{3}$. Let~$\mathbf{p} \in (\PP^3)^9$ be distinct points such that $p_{i} \in \ell_{j}$ if and only if $i \equiv j \;\mathrm{mod}\;3$. Then~$\mathbf{p} \in V_{3, 9}$. We claim that $\ell_{1} \cup \ell_{2} \cup \ell_{3}$ is the only degeneration of a twisted cubic containing all of the~$p_{i}$. Indeed, if $C$ is a rational curve containing all the~$p_{i}$, then $C$ meets $\PP^{2}$ at least 9 times. Because $\deg C = 3 < 9$, $C$ must have an irreducible component contained in~$\PP^{2}$. The only plane curve with degree~$\le 3$ passing through all the~$p_{i}$ is $\ell_{1} \cup \ell_{2} \cup \ell_{3}$. This verifies the claim.  

Now consider the total evaluation map $\nu \colon \overline{\rM}_{0, 9}(\PP^{3}, 3) \to V_{3, 9} \subseteq (\PP^3)^9$. The fiber~$\nu^{-1}(\mathbf{p})$ consists of three distinct points, each corresponding to a stable map $f_i \colon C = C_{1} \cup C_{2} \cup C_{3} \to \PP^{3}$, $i=1,2,3$, where $C$ is a chain of~$\PP^1$s with central component $C_{2}$, and $f_i(C_{2}) = \ell_{i}$ (this is explained in Remark~\ref{whythreepointsinfiber}). Since $\nu$ is a birational projective morphism, $V_{3, 9}$ is not normal by Zariski's Main Theorem~\cite[Corollary III.11.4]{Har77} because $\overline{\rM}_{0,n}(\PP^{3}, 3)$ is an integral scheme and the fiber~$\nu^{-1}(\mathbf{p})$ is disconnected. One can easily generalize this argument to all~$n \ge 9$.

For $n=8$, let $\mathbf{p}\in V_{3,8}$ be a degenerate configuration of eight points not lying on a conic (such a configuration exists by Proposition~\ref{prop:irreducible}). For the total evaluation map $\nu \colon \overline{\rM}_{0, 8}(\PP^{3}, 3) \to V_{3, 8} \subseteq (\PP^3)^8$, the fiber $\nu^{-1}(\mathbf{p})$ corresponds to degree~$3$ stable maps $f\colon C\rightarrow\mathbb{P}^3$ such that $f(C)$ passes through these eight points, and the only possibility for~$f(C)$ is to be a singular planar cubic. Since there are finitely many, but more than one, singular planar cubics passing through eight general points (there are 12 to be precise, by considering the discriminant~\cite[Appendix~A, Prop~1.1]{Sil09}), we conclude that the fiber~$\nu^{-1}(\mathbf{p})$ is disconnected and hence again by Zariski's Main Theorem that $V_{3,8}$ is not normal.
\end{proof}

\begin{remark}
\label{whythreepointsinfiber}
Let~$\mathbf{p}\in V_{3,9}$ as in the proof of Proposition~\ref{notnormalindegree3}. We provide an argument explaining why $\nu^{-1}(\mathbf{p})$ consists of the three claimed elements.

Let~$(g \colon X\rightarrow\mathbb{P}^3,x_1,\ldots,x_9)$ be a stable map parametrized by a point in~$\nu^{-1}(\mathbf{p})$. Then $X$ has at least three irreducible components $C_1,C_2,C_3$ mapping to~$\ell_1,\ell_2,\ell_3$ in some order. Assume by contradiction that there exists an irreducible component $C\subseteq X$ different from~$C_1,C_2,C_3$. Then $C$ must be contracted by~$g$, implying that $C$ has at least three special points by the stability condition.

Let $G$ be the dual graph of~$X$, which is a tree. Let $v$ be the vertex corresponding to~$C$. Observe that $v$ cannot be a leaf of~$G$, otherwise $C$ would have at most two special points (the unique node and possibly one marking). Hence, $G$ has at most three leaves, implying that the degree of~$v$ is either $2$ or $3$. If the degree is~$3$, then the images of $C_1,C_2,C_3$ under~$g$ have a common point, which is not. So $v$ has degree~$2$, but this violates the stability condition.
\end{remark}

		

\end{document}